\documentclass[11pt]{article}
\usepackage[a4paper]{anysize}\marginsize{3.5cm}{3.5cm}{1.3cm}{2cm}
\pdfpagewidth=\paperwidth \pdfpageheight=\paperheight
\usepackage{amsfonts,amssymb,amsthm,amsmath,eucal}
\usepackage{pgf}
\usepackage{bbm,array}
\usepackage{tikz} %Used for drawing picture in LaTeX
\usepackage{float}
\restylefloat{table}
\usepackage{subfigure}
\usepackage{caption}
\usepackage{enumerate}

\usetikzlibrary{arrows}
\let\oriLarge=\Large
\renewcommand\LARGE{\bfseries\oriLarge}
\renewcommand\Large{\addvspace{2\baselineskip}\centering}
\renewcommand\large{\addvspace{2\baselineskip}\centering}

\theoremstyle{plain}
\newtheorem{thm}{Theorem}[section]
\newtheorem{theorem}[thm]{Theorem}

\newtheorem{lemma}[thm]{Lemma}

\newtheorem{corollary}[thm]{Corollary}

\theoremstyle{definition}
\newtheorem{definition}[thm]{Definition}
\newtheorem{remark}[thm]{Remark}

\newtheorem{question}[thm]{Question}
\newtheorem{problem}[thm]{Problem}
\newtheorem{fact}[thm]{Fact}

\newtheorem{thevarthm}[thm]{\varthmname}

\newenvironment{varthm*}[1]{\trivlist\item[]{\bf #1.}\it}{\endtrivlist}

\renewcommand\ge{\geqslant}
\renewcommand\geq{\geqslant}
\renewcommand\le{\leqslant}
\renewcommand\leq{\leqslant}
\newcommand\longto{\longrightarrow}
\let\tilde=\widetilde
\newcommand\numequiv{\equiv_{\rm num}}
\newcommand\dsupseteq{{\displaystyle\supseteq}}
\newcommand\be{\begin{eqnarray*}}
\newcommand\ee{\end{eqnarray*}}
\newcommand\compact{\itemsep=0cm \parskip=0cm}
\newcommand{\dl}{\Delta_{Y_\bullet}(L)}
\newcommand{\inteps}{[0,\epsilon_C(L)]}
\newcommand\Q{\mathbb Q}
\newcommand\R{\mathbb R}
\newcommand\C{\mathbb C}
\newcommand\Z{\mathbb Z}
\newcommand\N{\mathbb N}
\newcommand\F{\mathbb F}
\newcommand\T{\mathbb T}
\newcommand\K{\mathbb K}
\renewcommand\P{\mathbb P}
\newcommand\FH{\mathbb H}
\newcommand\omegalogc{\Omega^1_X(\log C)}
\newcommand\legendre[2]{\left(\frac{#1}{#2}\right)}
\newcommand\calo{{\mathcal O}}
\newcommand\cali{{\mathcal I}}
\newcommand\calj{{\mathcal J}}
\newcommand\calk{{\mathcal K}}
\newcommand\calf{{\mathcal F}}
\newcommand\call{{\mathcal L}}
\newcommand\calm{{\mathcal M}}
\newcommand\frakm{{\mathfrak m}}
\newcommand\insn{I_{n,s,n}}
\newcommand\icsn{I_{c,s,n}}
\newcommand\icsc{I_{c,s,c}}
\newcommand\vnsn{V_{n,s,n}}
\newcommand\vcsn{V_{c,s,n}}
\newcommand\vcsc{V_{c,s,c}}
\def\field{\C}
\newcommand\newop[2]{\def#1{\mathop{\rm #2}\nolimits}}
\newop\log{log}
\newop\ord{ord}
\newop\Gal{Gal}
\newop\SL{SL}
\newop\GL{GL}
\newop\Bl{Bl}
\newop\mult{mult}
\newop\mass{mass}
\newop\div{div}
\newop\codim{codim}
\newop\sing{sing}
\newop\vdim{vdim}
\newop\edim{edim}
\newop\Ass{Ass}
\newop\size{size}
\newop\reg{reg}
\newop\areg{areg}
\newop\asreg{asreg}
\newop\satdeg{satdeg}
\newop\supp{supp}
\newop\gin{gin}
\newop\ini{in}
\newop\vol{vol}
\newop\sat{sat}
\newop\length{length}
\newop\depth{depth}
\newop\characteristic{char}
\newcommand\eqnref[1]{(\ref{#1})}
\newcommand{\olc}{\ensuremath{\Omega_X^1(\log C)}}
\newcommand\marginnote[1]{%
   {{\small\upshape\sffamily $\langle$...#1...$\rangle$}}
   \marginpar{$\leftarrow$\raggedright\scriptsize\sffamily #1}}
\newcommand\meta[1]{{\small\upshape\sffamily$\langle\!\langle\,$%
\marginpar[\hfill\raisebox{-0.75ex}{\Huge$\rightarrow$}]{\raisebox{-0.75ex}{\Huge$\leftarrow$}}%
   #1$\,\rangle\!\rangle$}}

\newcommand\ibul{I_{\bullet}}
\newcommand\gbul{G_{\bullet}}
\newcommand\fbul{\calf_{\bullet}}
\newcommand\ebul{E_{\bullet}}
\newcommand\vbul{V_{\bullet}}
\newcommand\ybul{Y_{\bullet}}
\newcommand\emin{e_{\min}}
\newcommand\emax{e_{\max}}
\newcommand\nuy{\nu_{\ybul}}
\newcommand\phifbul{\varphi_{\fbul}}
\newcommand\wtilde[1]{\widetilde{#1}}
\newcommand\restr[1]{\big|_{#1}}

\newcommand\subm{\underline{m}}
\newcommand\subn{\underline{n}}
\newcommand\one{\mathbbm{1}}
\newcommand\ldk{(\P^2,\calo_{\P^2}(d)\otimes\cali(kZ))}
\newcommand\ldmi{(\P^2,\calo_{\P^2}(d)\otimes\cali(\subm Z))}
\newcommand\ldthreem{(\P^2,\calo_{\P^2}(\widetilde{d}-3)\otimes\cali((\widetilde{\subm}-\one)Z))}
\newcommand\ldmij{(\P^2,\calo_{\P^2}(d-1)\otimes\cali((\subm-\one)Z))}
\newcommand\alphasubmn{\alpha_{\subm,\subn}}
\newcommand\mj{m^{(j)}}
\newcommand\nj{n^{(j)}}
\newcommand\mji{\mj_i}
\newcommand\submj{\subm^{(j)}}
\newcommand\subnj{\subn^{(j)}}
\newcommand\rru{\rule{1cm}{0cm}}
\newcommand\eps{\varepsilon}
\newcommand\lra{\longrightarrow}

\def\keywordname{{\bfseries Keywords}}%
\def\keywords#1{\par\addvspace\medskipamount{\rightskip=0pt plus1cm
\def\and{\ifhmode\unskip\nobreak\fi\ $\cdot$
}\noindent\keywordname\enspace\ignorespaces#1\par}}
\def\subclassname{{\bfseries Mathematics Subject Classification
(2000)}\enspace}
\def\subclass#1{\par\addvspace\medskipamount{\rightskip=0pt plus1cm
\def\and{\ifhmode\unskip\nobreak\fi\ $\cdot$
}\noindent\subclassname\ignorespaces#1\par}}

\def\sys{\mathcal{L}}
\newcommand\rounddown[1]{\left\lfloor#1\right\rfloor}

\definecolor{qqqqff}{rgb}{0,0,0}
\definecolor{uuuuuu}{rgb}{0,0,0}
\definecolor{zzttqq}{rgb}{0,0,0}
\definecolor{xdxdff}{rgb}{0,0,0}
\definecolor{ttqqqq}{rgb}{0.2,0.,0.}
\definecolor{uuuuuu}{rgb}{0.26666666666666666,0.26666666666666666,0.26666666666666666}
\definecolor{xdxdff}{rgb}{0.49019607843137253,0.49019607843137253,1.}
\definecolor{wwccqq}{rgb}{0.4,0.8,0.}
\definecolor{qqqqcc}{rgb}{0.,0.,0.8}
\definecolor{ffttww}{rgb}{1.,0.2,0.4}
\begin{document}

\author{M.~Lampa-Baczy\'nska, J.~Szpond\footnote{JS was partially supported by National Science Centre, Poland, grant 2014/15/B/ST1/02197}}
\title{From Pappus Theorem to parameter spaces of some extremal line point configurations and applications}
\date{\today}
\maketitle
\thispagestyle{empty}

\begin{abstract}

\subclass{52C35, 32S22, 14N20, 13F20}
   In the present work we study parameter spaces of two line point configurations
   introduced by B\"or\"oczky. These configurations are extremal from the point
   of view of Dirac-Motzkin Conjecture settled recently by Green and Tao \cite{GreTao13}.
   They have appeared also recently in commutative algebra in connection with
   the containment problem for symbolic and ordinary powers of homogeneous
   ideals \cite{DST13} and in algebraic geometry in considerations revolving
   around the Bounded Negativity Conjecture \cite{BNC}. Our main results are
   Theorem A and Theorem B. We show that the parameter space of what we call $B12$
   configurations is a three dimensional rational variety. As a consequence we
   derive the existence of a three dimensional family of rational $B12$ configurations.
   On the other hand the parameter space of $B15$ configurations is shown to be
   an elliptic curve with only finitely many rational points, all corresponding
   to degenerate configurations. Thus, somewhat surprisingly, we conclude that there are no
   rational $B15$ configurations.
\end{abstract}

%*****************************************************************************

\section{Introduction}
\label{intro}
   Line point configurations are a classical subject of study in geometry.
   One of motivations to study configurations of real lines is the following result.
\begin{varthm*}{Theorem}
   {\rm (Sylvester-Gallai)} Let $L_1,\ldots,L_d$ be finitely many mutually distinct lines
   in the real projective plane. Then
\begin{itemize}
   \item[a)] either all lines intersect in one point (the lines belong to one pencil);
   \item[b)] or there exists a point where exactly two of the given lines intersect.
\end{itemize}
\end{varthm*}
\begin{definition}[Ordinary point]
   A point where exactly two lines from a given configuration of lines $L_1,\ldots,L_d$
   meet is called an \emph{ordinary point} of this configuration.
\end{definition}
   The Diract Conjecture solved recently by Green and Tao in \cite[Theorem 1.2]{GreTao13}
   predicted that the number of ordinary points is bounded from below by $\lfloor d/2\rfloor$.
   A series of examples \cite[Proposition 2.1]{GreTao13} due essentially to B\"or\"oczky,
   shows that this lower bound cannot be improved in general. We denote by $B\mathrm{d}$
   the B\"or\"oczky example for $d$ lines. Thus e.g. $B12$ is a configuration
   of B\"or\"oczky containing $12$ lines.

   Recently B\"or\"oczky configurations have appeared in algebraic geometry in connection
   with the following problem due to Huneke
\begin{question}[Huneke]\label{que:Huneke}
   Let $I$ be the ideal of finitely many points in $\P^2(\K)$. Is then
   $$I^{(3)}\subset I^2?$$
\end{question}
   See \cite{SzeSzp16} for the notation, background, history and more results.
   A first counterexample to Question \ref{que:Huneke} has been announced
   by Dumnicki, Szemberg and Tutaj-Gasi\'nska in \cite{DST13}.
   The points in this counterexample arise as all intersection points
   of a certain configuration of complex lines. Shortly after \cite{DST13}
   it has been discovered in \cite{CGMLLPS2015} that intersection points of B\"or\"oczky
   configurations provide a counterexample in $\R[x_0,x_1,x_2]$.
   Harbourne and Seceleanu \cite{HarSec15} have discovered series
   of counterexamples over arbitrary finite fields.
   Finally it has been noticed in \cite[Figure 1]{DHNSST2015} that
   B\"or\"oczky construction for $12$ lines can be carried out over the field $\Q$.
   It has been expected that all B\"or\"oczky examples have a rational realization.
   In the present note we study parameter spaces for configurations $B12$ and $B15$.
   The main results are the following two theorems.
\begin{varthm*}{Theorem A}
   All $B12$ configurations form a $3$ dimensional family parametrized
   by a rational variety (an open set in $(\P^1(\R))^3$).
\end{varthm*}
\begin{varthm*}{Theorem B}
   All $B15$ configurations form a $1$ dimensional family parametrized by an elliptic curve.
\end{varthm*}
   Theorem B has the following, somewhat surprising, consequence.
\begin{corollary}\label{cor:corollary B}
   There is no $B15$ configuration over the rational numbers.
\end{corollary}
   As a consequence we conclude that (so far) there is just one rational
   counterexample to Huneke Question \ref{que:Huneke}.
   This is quite striking and justifies the following question.
\begin{problem}
   Construct more rational counterexamples to the containment problem.
\end{problem}

\section{A constructions of $12$ lines with $19$ triple points}

\subsection{Geometric construction}
   In this section we provide a direct geometric construction of B\"or\"oczky configuration
   of $12$ lines, which works in a projective plane defined over an arbitrary field with
   sufficiently many elements.

   The main auxiliary results which come into the argument are the following.
\begin{theorem}[Pappus Theorem]\label{thm: perspective}
If triangles $ABC$ and $DEF$ are perspective in two ways with
perspective centers $P$, $Q$, then there is also a third perspective
center $R$ (see Figure \ref{fig: perspective}).

\begin{figure}[h]
\centering
\begin{tikzpicture}[line cap=round,line join=round,>=triangle 45,x=0.6cm,y=0.6cm]
\clip(-2,-6) rectangle (6.7,1.6);
\draw [color=black,domain=-2:6.7] plot(\x,{(--8.35-4.16*\x)/0.15});
\draw [color=black,domain=-2:6.7] plot(\x,{(--0.08--3.75*\x)/-4.28});
\draw [color=black,domain=-2:6.7] plot(\x,{(--7.02--2.63*\x)/-4.24});
\draw [color=black,domain=-2:6.7] plot(\x,{(-16.47--2.48*\x)/3.79});
\draw [color=black,domain=-2:6.7] plot(\x,{(--8.29-3.7*\x)/0.86});
\draw [color=black,domain=-2:6.7] plot(\x,{(--13.22--1.65*\x)/-4.24});
\draw [color=black,domain=-2:6.7] plot(\x,{(--8.43-4.78*\x)/-0.79});
\draw [color=black,domain=-2:6.7] plot(\x,{(-11.56--2.95*\x)/3.03});
\draw [color=black,domain=-2:6.7] plot(\x,{(-22.7--1.88*\x)/4.75});
\begin{scriptsize}
\fill [color=black] (1.96,1.2) circle (1.5pt);
\draw[color=black] (2.25,1.3) node {$R$};
\fill [color=black] (6.35,-5.59) circle (1.5pt);
\draw[color=black] (6.53,-5.3) node {$Q$};
\fill [color=black] (2.08,-1.84) circle (1.5pt);
\draw[color=black] (2.3,-1.3) node {$F$};
\fill [color=black] (-1.68,-5.44) circle (1.5pt);
\draw[color=black] (-1.83,-5.14) node {$P$};
\fill [color=black] (2.83,-2.5) circle (1.5pt);
\draw[color=black] (3.01,-2.1) node {$C$};
\fill [color=black] (3.07,-3.56) circle (1.5pt);
\draw[color=black] (3.18,-3.25) node {$E$};
\fill [color=black] (2.11,-3.94) circle (1.5pt);
\draw[color=black] (1.8,-3.5) node {$A$};
\fill [color=black] (1.18,-3.58) circle (1.5pt);
\draw[color=black] (1.03,-3.1) node {$D$};
\fill [color=black] (1.36,-2.5) circle (1.5pt);
\draw[color=black] (1.22,-2.1) node {$B$};
\end{scriptsize}
\end{tikzpicture}
\caption{}
\label{fig: perspective}
\end{figure}

\end{theorem}
  The following version of Bezout's Theorem is taken from \cite{BS2015}.
\begin{theorem}\label{thm: Pascal}
If two projective curves $C$ and $D$ in $\mathbb P^2$ of degree $n$ intersect at
exactly $n^2$ points and if $n\cdot m$ of these points lie on irreducible curve $E$ of degree
$m < n$, then the remaining $n\cdot (n - m)$ points lie on curve of degree at most $n - m$.
\end{theorem}
   Now we are in the position to run our geometric construction.

   To begin with let $A$, $B$ and $C$ be non-collinear points in $\mathbb P^2$. Then choose three points $D$, $E$, $F$ on the lines $AC$, $AB$ and $BC$,
respectively, different from the points $A$, $B$, $C$. The points $D$, $E$, $F$ will be parameters of our construction. Then we define consecutively the following
$6$ points.
 \begin{equation}\label{eq:points1}
   \begin{array}{ccc}
    G & = & AF \cap BD,\\
    H & = & BD \cap EF,\\
    I & = & BC \cap ED,\\
    J & = & AF \cap ED,\\
    K & = & AC \cap EF,\\
    L & = & HJ \cap IK.
   \end{array}
   \end{equation}
At this stage an additional incidence comes into the picture.
\begin{lemma}\label{lem: collinear1}
The line $AB$ passes through the point $L$.
\end{lemma}

\begin{proof}
The incidences coming from the construction so far are indicated in Figure \ref{fig: point L}.
\begin{figure}[H]
\centering
\begin{tikzpicture}[line cap=round,line join=round,>=triangle 45,x=1.0cm,y=1.0cm,scale=0.7]
%\clip(-6.233333333333335,-2.9) rectangle (14.726666666666665,9.08);
\clip(-3.233333333333335,-2.3) rectangle (9.726666666666665,9.08);
\draw [domain=-2.5:14.726666666666665] plot(\x,{(--9.541866666666666-1.54*\x)/3.26});%prosta DCKAN
\draw [domain=2.5:3.5] plot(\x,{(--5.9436-2.16*\x)/-0.02}); %prosta BEATSL
\draw [domain=-6.233333333333335:9.4] plot(\x,{(-3.4741333333333326--0.62*\x)/3.28}); %prosta FOBIC
\draw [domain=-6.233333333333335:7.5] plot(\x,{(--0.4227038325134407-3.1112790298592836*\x)/-5.052572932158789});
%prosta FGJAP
\draw [domain=-6.233333333333335:8.75] plot(\x,{(-3.5117772595537353--1.751710360158245*\x)/5.039984333365261});
%prosta FQHEK
\draw [domain=-6.233333333333335:14.726666666666665] plot(\x,{(-1.5778755345053685-0.5625375291001036*\x)/5.783293730432688});
%prosta DMBHG
\draw [domain=-6.233333333333335:14.726666666666665] plot(\x,{(--5.257943361515681-1.3629688593990654*\x)/5.775882329226216});
%prosta DRIEJ
\draw [domain=-6.233333333333335:14.726666666666665] plot(\x,{(--1.0317160740424687-1.0004833741555315*\x)/-0.21392961431237845});
%prosta LNJHO
\draw [domain=-6.233333333333335:14.726666666666665] plot(\x,{(--4.698918506243937-1.0224659349540723*\x)/0.21471685702482635});
%prosta LPKIM
%\draw [domain=-6.233333333333335:5.2] plot(\x,{(-0.21610909610215334-2.4630053806498506*\x)/-1.7703517483201314});
%prosta SNGQ
%\draw [domain=0:14.726666666666665] plot(\x,{(--14.32886322749831-2.3518024770562342*\x)/1.9416704054925944});
%prosta SPCR
%\draw [domain=-6.233333333333335:14.726666666666665] plot(\x,{(-6.944919157563774--0.3144529324436316*\x)/7.418083235349504});
%prosta QOMR
%\draw [domain=-6.233333333333335:6] plot(\x,{(--2.8210648374795815--3.7072566054013603*\x)/3.790958064452835});
%prosta FNT
%\draw [domain=-1:14.726666666666665] plot(\x,{(-25.247040507035813--3.534340006156338*\x)/-4.444964135925282});
%prosta DPT
\begin{scriptsize}
\draw [fill=black] (2.7666666666666653,1.62) circle (1.5pt); %punkt A
\draw[color=black] (3.166666666666665,1.65) node {$A$};
\draw [fill=black] (2.7466666666666653,-0.54) circle (1.5pt); %punkt B
\draw[color=black] (2.8866666666666654,-0.3) node {$B$};
%\draw [fill=black] (6.0266666666666655,0.08) circle (1.5pt); %punkt C
%\draw[color=black] (6.166666666666665,0.36) node {$C$};
\draw [fill=black] (8.529960397099353,-1.1025375291001036) circle (1.5pt); %punkt D
\draw[color=black] (8.666666666666666,-0.92) node {$D$};
%\draw [fill=black] (2.754078067873137,0.2604313302989618) circle (1.5pt); %punkt E
%\draw[color=black] (2.8866666666666654,0) node {$E$};
\draw [fill=black] (-2.285906265492124,-1.4912790298592833) circle (1.5pt); %punkt F
\draw[color=black] (-2.373333333333335,-1.29) node {$F$};
%\draw [fill=black] (-0.2652999493594206,-0.24702780510777403) circle (1.5pt); %punkt G
%\draw[color=black] (-0.35333333333333486,0.06) node {$G$};
\draw [fill=black] (0.953056289709474,-0.3655365892428371) circle (1.5pt); %punkt H
\draw[color=black] (0.8066666666666651,-0.12) node {$H$};
\draw [fill=black] (4.634148432252167,-0.18321991016371605) circle (1.5pt); %punkt I
\draw[color=black] (4.766666666666665,0.1) node {$I$};
\draw [fill=black] (1.1669859040218524,0.6349467849126943) circle (1.5pt); %punkt J
\draw[color=black] (1.3666666666666651,0.97) node {$J$};
\draw [fill=black] (4.419431575227341,0.8392460247903564) circle (1.5pt); %punkt K
\draw[color=black] (4.306666666666665,0.55) node {$K$};
\draw [fill=black] (2.8295388752500106,8.410198527001244) circle (1.5pt); %punkt L
\draw[color=black] (3.046666666666665,8.42) node {$L$};
%\draw [fill=black] (4.74999280834317,-0.7348623380116698) circle (1.5pt); %punkt M
%\draw[color=black] (4.886666666666665,-0.56) node {$M$};
%\draw [fill=black] (1.5050517989607108,2.2159775755420767) circle (1.5pt); %punkt N
%\draw[color=black] (1.3266666666666653,2.5) node {$N$};
%\draw [fill=black] (0.838631756054356,-0.9006651355425708) circle (1.5pt); %punkt O
%\draw[color=black] (0.9866666666666652,-0.72) node {$O$};
%\draw [fill=black] (4.084996261174071,2.4318024770562343) circle (1.5pt); %punkt P
%\draw[color=black] (4.226666666666665,2.72) node {$P$};
%\draw [fill=black] (-0.7845770053117687,-0.9694730389841828) circle (1.5pt); %punkt Q
%\draw[color=black] (-0.8133333333333348,-0.78) node {$Q$};
%\draw [fill=black] (6.633506230037735,-0.6550201065405512) circle (1.5pt); %punkt R
%\draw[color=black] (6.666666666666665,-0.48) node {$R$};
%\draw [fill=black] (2.7887211584967786,4.001885117652153) circle (1.5pt); %punkt S
%\draw[color=black] (2.8866666666666654,4.41) node {$S$};
%\draw [fill=uuuuuu] (2.783763484103525,3.4664562831808063) circle (1.5pt); %punkt T
%\draw[color=uuuuuu] (2.8866666666666654,3.14) node {$T$};
\end{scriptsize}
\end{tikzpicture}
   \caption{}
   \label{fig: point L}
   \end{figure}

Since triangles $BJK$ and $AHI$ are perspective with perspective centers $D$, $F$, Theorem \ref{thm: perspective} yields that there is the third perspective
center $L= HJ\cap IK\cap AB$. In particular, the points $A$, $B$ and $L$ are collinear.
\end{proof}

Then we define the remaining $7$ points:
\begin{equation}\label{eq:points2}
   \begin{array}{ccc}
    M & = & BD \cap IK,\\
    N & = & AC \cap HJ,\\
    O & = & HJ \cap BC,\\
    P & = & AF \cap IK,\\
    Q & = & EF \cap NG,\\
    R & = & DE \cap CP,\\
    S & = & CP \cap NG.
   \end{array}
   \end{equation}

%   \begin{equation*}\label{eq:points2}
%   \begin{array}{llll}
%   M = BD \cap IK,          & N = CD \cap HJ,        & O = HJ \cap BC,      & P = AF \cap IK,\\
%   Q = EF \cap NG,          & R = DE \cap CP,        & S = CP \cap NG.
%   \end{array}
%   \end{equation*}

Here we claim two additional collinearities.
\begin{lemma}\label{lem: collinear2}
From the above assumptions it follows that
the points $A$, $B$ and $S$ are collinear.
\end{lemma}

\begin{proof}
In order to prove the collinearity of points $A$, $B$ and $S$ we introduce an extra point $T= FN\cap DP$. \\
\textbf{Claim.} The points $T$, $L$ and $E$ are collinear.\\
Note that this implies the collinearity of $T$ with $A$ and $B$ as well. Figure \ref{fig: point T} contains points relevant for the proof of the Claim.
\begin{figure}[H]
\centering
\begin{tikzpicture}[line cap=round,line join=round,>=triangle 45,x=1.0cm,y=1.0cm,scale=0.7]
%\clip(-6.233333333333335,-2.9) rectangle (14.726666666666665,9.08);
\clip(-3.233333333333335,-2.3) rectangle (9.726666666666665,9.08);
\draw [domain=-2.5:14.726666666666665] plot(\x,{(--9.541866666666666-1.54*\x)/3.26});%prosta DCKAN
\draw [domain=2.5:3.5] plot(\x,{(--5.9436-2.16*\x)/-0.02}); %prosta BEATSL
%\draw [domain=-6.233333333333335:9.4] plot(\x,{(-3.4741333333333326--0.62*\x)/3.28}); %prosta FOBIC
\draw [domain=-6.233333333333335:7.5] plot(\x,{(--0.4227038325134407-3.1112790298592836*\x)/-5.052572932158789});
%prosta FGJAP
\draw [domain=-6.233333333333335:8.75] plot(\x,{(-3.5117772595537353--1.751710360158245*\x)/5.039984333365261});
%prosta FQHEK
%\draw [domain=-6.233333333333335:14.726666666666665] plot(\x,{(-1.5778755345053685-0.5625375291001036*\x)/5.783293730432688});
%prosta DMBHG
\draw [domain=-6.233333333333335:14.726666666666665] plot(\x,{(--5.257943361515681-1.3629688593990654*\x)/5.775882329226216});
%prosta DRIEJ
\draw [domain=-6.233333333333335:14.726666666666665] plot(\x,{(--1.0317160740424687-1.0004833741555315*\x)/-0.21392961431237845});
%prosta LNJHO
\draw [domain=-6.233333333333335:14.726666666666665] plot(\x,{(--4.698918506243937-1.0224659349540723*\x)/0.21471685702482635});
%prosta LPKIM
%\draw [domain=-6.233333333333335:5.2] plot(\x,{(-0.21610909610215334-2.4630053806498506*\x)/-1.7703517483201314});
%prosta SNGQ
%\draw [domain=0:14.726666666666665] plot(\x,{(--14.32886322749831-2.3518024770562342*\x)/1.9416704054925944});
%prosta SPCR
%\draw [domain=-6.233333333333335:14.726666666666665] plot(\x,{(-6.944919157563774--0.3144529324436316*\x)/7.418083235349504});
%prosta QOMR
\draw [domain=-6.233333333333335:6] plot(\x,{(--2.8210648374795815--3.7072566054013603*\x)/3.790958064452835});
%prosta FNT
\draw [domain=-1:14.726666666666665] plot(\x,{(-25.247040507035813--3.534340006156338*\x)/-4.444964135925282});
%prosta DPT
\begin{scriptsize}
%\draw [fill=black] (2.7666666666666653,1.62) circle (1.5pt); %punkt A
%\draw[color=black] (3.166666666666665,1.65) node {$A$};
%\draw [fill=black] (2.7466666666666653,-0.54) circle (1.5pt); %punkt B
%\draw[color=black] (2.8866666666666654,-0.3) node {$B$};
%\draw [fill=black] (6.0266666666666655,0.08) circle (1.5pt); %punkt C
%\draw[color=black] (6.166666666666665,0.36) node {$C$};
\draw [fill=black] (8.529960397099353,-1.1025375291001036) circle (1.5pt); %punkt D
\draw[color=black] (8.666666666666666,-0.92) node {$D$};
\draw [fill=black] (2.754078067873137,0.2604313302989618) circle (1.5pt); %punkt E
\draw[color=black] (2.8866666666666654,0) node {$E$};
\draw [fill=black] (-2.285906265492124,-1.4912790298592833) circle (1.5pt); %punkt F
\draw[color=black] (-2.373333333333335,-1.29) node {$F$};
%\draw [fill=black] (-0.2652999493594206,-0.24702780510777403) circle (1.5pt); %punkt G
%\draw[color=black] (-0.35333333333333486,0.06) node {$G$};
%\draw [fill=black] (0.953056289709474,-0.3655365892428371) circle (1.5pt); %punkt H
%\draw[color=black] (0.8066666666666651,-0.12) node {$H$};
%\draw [fill=black] (4.634148432252167,-0.18321991016371605) circle (1.5pt); %punkt I
%\draw[color=black] (4.766666666666665,0.1) node {$I$};
\draw [fill=black] (1.1669859040218524,0.6349467849126943) circle (1.5pt); %punkt J
\draw[color=black] (1.3666666666666651,0.97) node {$J$};
\draw [fill=black] (4.419431575227341,0.8392460247903564) circle (1.5pt); %punkt K
\draw[color=black] (4.306666666666665,0.55) node {$K$};
\draw [fill=black] (2.8295388752500106,8.410198527001244) circle (1.5pt); %punkt L
\draw[color=black] (3.046666666666665,8.42) node {$L$};
%\draw [fill=black] (4.74999280834317,-0.7348623380116698) circle (1.5pt); %punkt M
%\draw[color=black] (4.886666666666665,-0.56) node {$M$};
\draw [fill=black] (1.5050517989607108,2.2159775755420767) circle (1.5pt); %punkt N
\draw[color=black] (1.3266666666666653,2.5) node {$N$};
%\draw [fill=black] (0.838631756054356,-0.9006651355425708) circle (1.5pt); %punkt O
%\draw[color=black] (0.9866666666666652,-0.72) node {$O$};
\draw [fill=black] (4.084996261174071,2.4318024770562343) circle (1.5pt); %punkt P
\draw[color=black] (4.226666666666665,2.72) node {$P$};
%\draw [fill=black] (-0.7845770053117687,-0.9694730389841828) circle (1.5pt); %punkt Q
%\draw[color=black] (-0.8133333333333348,-0.78) node {$Q$};
%\draw [fill=black] (6.633506230037735,-0.6550201065405512) circle (1.5pt); %punkt R
%\draw[color=black] (6.666666666666665,-0.48) node {$R$};
%\draw [fill=black] (2.7887211584967786,4.001885117652153) circle (1.5pt); %punkt S
%\draw[color=black] (2.8866666666666654,4.41) node {$S$};
\draw [fill=uuuuuu] (2.783763484103525,3.4664562831808063) circle (1.5pt); %punkt T
\draw[color=uuuuuu] (2.8866666666666654,3.14) node {$T$};
\end{scriptsize}
\end{tikzpicture}
   \caption{}
   \label{fig: point T}
   \end{figure}

   Pappus Theorem applied to triangles $NPE$ and $TJK$ yields that points $L$, $T$ and $E$ are collinear.

By the same token we show the collinearity of points $T$, $S$ and $B$. All relevant points are marked on Figure \ref{fig: point S}. We leave the exact argument
to the reader.
\begin{figure}[H]
\centering
\begin{tikzpicture}[line cap=round,line join=round,>=triangle 45,x=1.0cm,y=1.0cm,scale=0.7]
%\clip(-6.233333333333335,-2.9) rectangle (14.726666666666665,9.08);
\clip(-3.233333333333335,-2.3) rectangle (9.726666666666665,6.08);
\draw [domain=-2.5:14.726666666666665] plot(\x,{(--9.541866666666666-1.54*\x)/3.26});%prosta DCKAN
\draw [domain=2.5:3.5] plot(\x,{(--5.9436-2.16*\x)/-0.02}); %prosta BEATSL
\draw [domain=-6.233333333333335:9.4] plot(\x,{(-3.4741333333333326--0.62*\x)/3.28}); %prosta FOBIC
\draw [domain=-6.233333333333335:7.5] plot(\x,{(--0.4227038325134407-3.1112790298592836*\x)/-5.052572932158789});
%prosta FGJAP
%\draw [domain=-6.233333333333335:8.75] plot(\x,{(-3.5117772595537353--1.751710360158245*\x)/5.039984333365261});
%prosta FQHEK
\draw [domain=-6.233333333333335:14.726666666666665] plot(\x,{(-1.5778755345053685-0.5625375291001036*\x)/5.783293730432688});
%prosta DMBHG
%\draw [domain=-6.233333333333335:14.726666666666665] plot(\x,{(--5.257943361515681-1.3629688593990654*\x)/5.775882329226216});
%prosta DRIEJ
%\draw [domain=-6.233333333333335:14.726666666666665] plot(\x,{(--1.0317160740424687-1.0004833741555315*\x)/-0.21392961431237845});
%prosta LNJHO
%\draw [domain=-6.233333333333335:14.726666666666665] plot(\x,{(--4.698918506243937-1.0224659349540723*\x)/0.21471685702482635});
%prosta LPKIM
\draw [domain=-6.233333333333335:5.2] plot(\x,{(-0.21610909610215334-2.4630053806498506*\x)/-1.7703517483201314});
%prosta SNGQ
\draw [domain=0:14.726666666666665] plot(\x,{(--14.32886322749831-2.3518024770562342*\x)/1.9416704054925944});
%prosta SPCR
%\draw [domain=-6.233333333333335:14.726666666666665] plot(\x,{(-6.944919157563774--0.3144529324436316*\x)/7.418083235349504});
%prosta QOMR
\draw [domain=-6.233333333333335:6] plot(\x,{(--2.8210648374795815--3.7072566054013603*\x)/3.790958064452835});
%prosta FNT
\draw [domain=-1:14.726666666666665] plot(\x,{(-25.247040507035813--3.534340006156338*\x)/-4.444964135925282});
%prosta DPT
\begin{scriptsize}
%\draw [fill=black] (2.7666666666666653,1.62) circle (1.5pt); %punkt A
%\draw[color=black] (3.166666666666665,1.65) node {$A$};
\draw [fill=black] (2.7466666666666653,-0.54) circle (1.5pt); %punkt B
\draw[color=black] (2.8866666666666654,-0.3) node {$B$};
\draw [fill=black] (6.0266666666666655,0.08) circle (1.5pt); %punkt C
\draw[color=black] (6.166666666666665,0.36) node {$C$};
\draw [fill=black] (8.529960397099353,-1.1025375291001036) circle (1.5pt); %punkt D
\draw[color=black] (8.666666666666666,-0.92) node {$D$};
%\draw [fill=black] (2.754078067873137,0.2604313302989618) circle (1.5pt); %punkt E
%\draw[color=black] (2.8866666666666654,0) node {$E$};
\draw [fill=black] (-2.285906265492124,-1.4912790298592833) circle (1.5pt); %punkt F
\draw[color=black] (-2.373333333333335,-1.29) node {$F$};
\draw [fill=black] (-0.2652999493594206,-0.24702780510777403) circle (1.5pt); %punkt G
\draw[color=black] (-0.35333333333333486,0.06) node {$G$};
%\draw [fill=black] (0.953056289709474,-0.3655365892428371) circle (1.5pt); %punkt H
%\draw[color=black] (0.8066666666666651,-0.12) node {$H$};
%\draw [fill=black] (4.634148432252167,-0.18321991016371605) circle (1.5pt); %punkt I
%\draw[color=black] (4.766666666666665,0.1) node {$I$};
%\draw [fill=black] (1.1669859040218524,0.6349467849126943) circle (1.5pt); %punkt J
%\draw[color=black] (1.3666666666666651,0.97) node {$J$};
%\draw [fill=black] (4.419431575227341,0.8392460247903564) circle (1.5pt); %punkt K
%\draw[color=black] (4.306666666666665,0.55) node {$K$};
%\draw [fill=black] (2.8295388752500106,8.410198527001244) circle (1.5pt); %punkt L
%\draw[color=black] (3.046666666666665,8.42) node {$L$};
%\draw [fill=black] (4.74999280834317,-0.7348623380116698) circle (1.5pt); %punkt M
%\draw[color=black] (4.886666666666665,-0.56) node {$M$};
\draw [fill=black] (1.5050517989607108,2.2159775755420767) circle (1.5pt); %punkt N
\draw[color=black] (1.3266666666666653,2.5) node {$N$};
%\draw [fill=black] (0.838631756054356,-0.9006651355425708) circle (1.5pt); %punkt O
%\draw[color=black] (0.9866666666666652,-0.72) node {$O$};
\draw [fill=black] (4.084996261174071,2.4318024770562343) circle (1.5pt); %punkt P
\draw[color=black] (4.226666666666665,2.72) node {$P$};
%\draw [fill=black] (-0.7845770053117687,-0.9694730389841828) circle (1.5pt); %punkt Q
%\draw[color=black] (-0.8133333333333348,-0.78) node {$Q$};
%\draw [fill=black] (6.633506230037735,-0.6550201065405512) circle (1.5pt); %punkt R
%\draw[color=black] (6.666666666666665,-0.48) node {$R$};
\draw [fill=black] (2.7887211584967786,4.001885117652153) circle (1.5pt); %punkt S
\draw[color=black] (2.8866666666666654,4.41) node {$S$};
\draw [fill=uuuuuu] (2.783763484103525,3.4664562831808063) circle (1.5pt); %punkt T
\draw[color=uuuuuu] (2.8866666666666654,3.14) node {$T$};
\end{scriptsize}
\end{tikzpicture}
   \caption{}
   \label{fig: point S}
   \end{figure}
We conclude that the  points $A$, $B$ and $S$ are collinear as asserted in the Lemma.
\end{proof}

\begin{lemma}\label{lem: collinear3}
In the construction above the points $M$, $O$, $Q$ and $R$ are collinear.
\end{lemma}
\begin{proof}
To prove the second collinearity (i.e. of points $O$, $Q$, $R$) we
take two reducible curves $\alpha=KN\cup FG\cup BL$ and
$\beta=KL\cup GN\cup BF$ of degree $3$. They intersect at $9$ points
and since points $C$, $P$, $S$ lie on a line, Theorem \ref{thm:
Pascal} yields that the remaining $6$ points: $B$, $F$, $G$, $K$,
$L$, $N$ lie on a conic, see Figure \ref{fig: conic1}.

\begin{figure}[H]
\centering
\begin{tikzpicture}[line cap=round,line join=round,>=triangle 45,x=1.0cm,y=1.0cm,scale=0.7]
%\clip(-6.233333333333335,-2.9) rectangle (14.726666666666665,9.08);
\clip(-3.233333333333335,-2.3) rectangle (9.726666666666665,9.08);
\draw [color=qqqqcc,domain=-2.5:14.726666666666665] plot(\x,{(--9.541866666666666-1.54*\x)/3.26});%prosta DCKAN
\draw [color=qqqqcc,domain=2.5:3.5] plot(\x,{(--5.9436-2.16*\x)/-0.02}); %prosta BEATSL
\draw [color=wwccqq,domain=-6.233333333333335:9.4] plot(\x,{(-3.4741333333333326--0.62*\x)/3.28}); %prosta FOBIC
\draw [color=qqqqcc,domain=-6.233333333333335:7.5] plot(\x,{(--0.4227038325134407-3.1112790298592836*\x)/-5.052572932158789});
%prosta FGJAP
%\draw [domain=-6.233333333333335:8.75] plot(\x,{(-3.5117772595537353--1.751710360158245*\x)/5.039984333365261});
%prosta FQHEK
%\draw [domain=-6.233333333333335:14.726666666666665] plot(\x,{(-1.5778755345053685-0.5625375291001036*\x)/5.783293730432688});
%prosta DMBHG
%\draw [domain=-6.233333333333335:14.726666666666665] plot(\x,{(--5.257943361515681-1.3629688593990654*\x)/5.775882329226216});
%prosta DRIEJ
%\draw [domain=-6.233333333333335:14.726666666666665] plot(\x,{(--1.0317160740424687-1.0004833741555315*\x)/-0.21392961431237845});
%prosta LNJHO
\draw [color=wwccqq,domain=-6.233333333333335:14.726666666666665] plot(\x,{(--4.698918506243937-1.0224659349540723*\x)/0.21471685702482635});
%prosta LPKIM
\draw [color=wwccqq,domain=-6.233333333333335:5.2] plot(\x,{(-0.21610909610215334-2.4630053806498506*\x)/-1.7703517483201314});
%prosta SNGQ
\draw [domain=0:14.726666666666665] plot(\x,{(--14.32886322749831-2.3518024770562342*\x)/1.9416704054925944});
%prosta SPCR
%\draw [domain=-6.233333333333335:14.726666666666665] plot(\x,{(-6.944919157563774--0.3144529324436316*\x)/7.418083235349504});
%prosta QOMR
%\draw [domain=-6.233333333333335:6] plot(\x,{(--2.8210648374795815--3.7072566054013603*\x)/3.790958064452835});
%prosta FNT
%\draw [domain=-1:14.726666666666665] plot(\x,{(-25.247040507035813--3.534340006156338*\x)/-4.444964135925282});
%prosta DPT
\begin{scriptsize}
%\draw [fill=black] (2.7666666666666653,1.62) circle (1.5pt); %punkt A
%\draw[color=black] (3.166666666666665,1.65) node {$A$};
\draw [fill=black] (2.7466666666666653,-0.54) circle (1.5pt); %punkt B
\draw[color=black] (2.8866666666666654,-0.3) node {$B$};
\draw [fill=black] (6.0266666666666655,0.08) circle (1.5pt); %punkt C
\draw[color=black] (6.166666666666665,0.36) node {$C$};
%\draw [fill=black] (8.529960397099353,-1.1025375291001036) circle (1.5pt); %punkt D
%\draw[color=black] (8.666666666666666,-0.92) node {$D$};
%\draw [fill=black] (2.754078067873137,0.2604313302989618) circle (1.5pt); %punkt E
%\draw[color=black] (2.8866666666666654,0) node {$E$};
\draw [fill=black] (-2.285906265492124,-1.4912790298592833) circle (1.5pt); %punkt F
\draw[color=black] (-2.373333333333335,-1.29) node {$F$};
\draw [fill=black] (-0.2652999493594206,-0.24702780510777403) circle (1.5pt); %punkt G
\draw[color=black] (-0.35333333333333486,0.06) node {$G$};
%\draw [fill=black] (0.953056289709474,-0.3655365892428371) circle (1.5pt); %punkt H
%\draw[color=black] (0.8066666666666651,-0.12) node {$H$};
%\draw [fill=black] (4.634148432252167,-0.18321991016371605) circle (1.5pt); %punkt I
%\draw[color=black] (4.766666666666665,0.1) node {$I$};
%\draw [fill=black] (1.1669859040218524,0.6349467849126943) circle (1.5pt); %punkt J
%\draw[color=black] (1.3666666666666651,0.97) node {$J$};
\draw [fill=black] (4.419431575227341,0.8392460247903564) circle (1.5pt); %punkt K
\draw[color=black] (4.306666666666665,0.55) node {$K$};
\draw [fill=black] (2.8295388752500106,8.410198527001244) circle (1.5pt); %punkt L
\draw[color=black] (3.046666666666665,8.42) node {$L$};
%\draw [fill=black] (4.74999280834317,-0.7348623380116698) circle (1.5pt); %punkt M
%\draw[color=black] (4.886666666666665,-0.56) node {$M$};
\draw [fill=black] (1.5050517989607108,2.2159775755420767) circle (1.5pt); %punkt N
\draw[color=black] (1.3266666666666653,2.5) node {$N$};
%\draw [fill=black] (0.838631756054356,-0.9006651355425708) circle (1.5pt); %punkt O
%\draw[color=black] (0.9866666666666652,-0.72) node {$O$};
\draw [fill=black] (4.084996261174071,2.4318024770562343) circle (1.5pt); %punkt P
\draw[color=black] (4.226666666666665,2.72) node {$P$};
%\draw [fill=black] (-0.7845770053117687,-0.9694730389841828) circle (1.5pt); %punkt Q
%\draw[color=black] (-0.8133333333333348,-0.78) node {$Q$};
%\draw [fill=black] (6.633506230037735,-0.6550201065405512) circle (1.5pt); %punkt R
%\draw[color=black] (6.666666666666665,-0.48) node {$R$};
\draw [fill=black] (2.7887211584967786,4.001885117652153) circle (1.5pt); %punkt S
\draw[color=black] (2.8866666666666654,4.41) node {$S$};
%\draw [fill=uuuuuu] (2.783763484103525,3.4664562831808063) circle (1.5pt); %punkt T
%\draw[color=uuuuuu] (2.8866666666666654,3.14) node {$T$};
\end{scriptsize}
\end{tikzpicture}
   \caption{}
   \label{fig: conic1}
   \end{figure}

Now we will show that the points $M$, $O$ and $Q$ lie on one line. In order to prove it we take the reducible cubics $\alpha= KL\cup BF\cup GN$ and
$\beta=BG\cup LN\cup FK$. They intersect in $9$ points and from the above statement we know that points  $B$, $F$, $G$, $K$, $L$, $N$ lie on a conic, hence the
remaining three points $M$, $O$ and $Q$ are collinear by Theorem \ref{thm: Pascal}, see Figure \ref{fig: conic2}.

\begin{figure}[H]
\centering
\begin{tikzpicture}[line cap=round,line join=round,>=triangle 45,x=1.0cm,y=1.0cm,scale=0.7]
%\clip(-6.233333333333335,-2.9) rectangle (14.726666666666665,9.08);
\clip(-3.233333333333335,-2.3) rectangle (9.726666666666665,9.08);
%\draw [color=qqqqcc,domain=-2.5:14.726666666666665] plot(\x,{(--9.541866666666666-1.54*\x)/3.26});%prosta DCKAN
%\draw [color=qqqqcc,domain=2.5:3.5] plot(\x,{(--5.9436-2.16*\x)/-0.02}); %prosta BEATSL
\draw [color=wwccqq,domain=-6.233333333333335:9.4] plot(\x,{(-3.4741333333333326--0.62*\x)/3.28}); %prosta FOBIC
%\draw [color=qqqqcc,domain=-6.233333333333335:7.5] plot(\x,{(--0.4227038325134407-3.1112790298592836*\x)/-5.052572932158789});
%prosta FGJAP
\draw [color=ffttww,domain=-6.233333333333335:8.75] plot(\x,{(-3.5117772595537353--1.751710360158245*\x)/5.039984333365261});
%prosta FQHEK
\draw [color=ffttww,domain=-6.233333333333335:14.726666666666665] plot(\x,{(-1.5778755345053685-0.5625375291001036*\x)/5.783293730432688});
%prosta DMBHG
%\draw [domain=-6.233333333333335:14.726666666666665] plot(\x,{(--5.257943361515681-1.3629688593990654*\x)/5.775882329226216});
%prosta DRIEJ
\draw [color=ffttww,domain=-6.233333333333335:14.726666666666665] plot(\x,{(--1.0317160740424687-1.0004833741555315*\x)/-0.21392961431237845});
%prosta LNJHO
\draw [color=wwccqq,domain=-6.233333333333335:14.726666666666665] plot(\x,{(--4.698918506243937-1.0224659349540723*\x)/0.21471685702482635});
%prosta LPKIM
\draw [color=wwccqq,domain=-6.233333333333335:5.2] plot(\x,{(-0.21610909610215334-2.4630053806498506*\x)/-1.7703517483201314});
%prosta SNGQ
%\draw [domain=0:14.726666666666665] plot(\x,{(--14.32886322749831-2.3518024770562342*\x)/1.9416704054925944});
%prosta SPCR
\draw [domain=-6.233333333333335:14.726666666666665] plot(\x,{(-6.944919157563774--0.3144529324436316*\x)/7.418083235349504});
%prosta QOMR
%\draw [domain=-6.233333333333335:6] plot(\x,{(--2.8210648374795815--3.7072566054013603*\x)/3.790958064452835});
%prosta FNT
%\draw [domain=-1:14.726666666666665] plot(\x,{(-25.247040507035813--3.534340006156338*\x)/-4.444964135925282});
%prosta DPT
\begin{scriptsize}
%\draw [fill=black] (2.7666666666666653,1.62) circle (1.5pt); %punkt A
%\draw[color=black] (3.166666666666665,1.65) node {$A$};
\draw [fill=black] (2.7466666666666653,-0.54) circle (1.5pt); %punkt B
\draw[color=black] (2.8866666666666654,-0.3) node {$B$};
%\draw [fill=black] (6.0266666666666655,0.08) circle (1.5pt); %punkt C
%\draw[color=black] (6.166666666666665,0.36) node {$C$};
%\draw [fill=black] (8.529960397099353,-1.1025375291001036) circle (1.5pt); %punkt D
%\draw[color=black] (8.666666666666666,-0.92) node {$D$};
%\draw [fill=black] (2.754078067873137,0.2604313302989618) circle (1.5pt); %punkt E
%\draw[color=black] (2.8866666666666654,0) node {$E$};
\draw [fill=black] (-2.285906265492124,-1.4912790298592833) circle (1.5pt); %punkt F
\draw[color=black] (-2.373333333333335,-1.29) node {$F$};
\draw [fill=black] (-0.2652999493594206,-0.24702780510777403) circle (1.5pt); %punkt G
\draw[color=black] (-0.35333333333333486,0.06) node {$G$};
%\draw [fill=black] (0.953056289709474,-0.3655365892428371) circle (1.5pt); %punkt H
%\draw[color=black] (0.8066666666666651,-0.12) node {$H$};
%\draw [fill=black] (4.634148432252167,-0.18321991016371605) circle (1.5pt); %punkt I
%\draw[color=black] (4.766666666666665,0.1) node {$I$};
%\draw [fill=black] (1.1669859040218524,0.6349467849126943) circle (1.5pt); %punkt J
%\draw[color=black] (1.3666666666666651,0.97) node {$J$};
\draw [fill=black] (4.419431575227341,0.8392460247903564) circle (1.5pt); %punkt K
\draw[color=black] (4.306666666666665,0.55) node {$K$};
\draw [fill=black] (2.8295388752500106,8.410198527001244) circle (1.5pt); %punkt L
\draw[color=black] (3.046666666666665,8.42) node {$L$};
\draw [fill=black] (4.74999280834317,-0.7348623380116698) circle (1.5pt); %punkt M
\draw[color=black] (4.886666666666665,-0.56) node {$M$};
\draw [fill=black] (1.5050517989607108,2.2159775755420767) circle (1.5pt); %punkt N
\draw[color=black] (1.3266666666666653,2.5) node {$N$};
\draw [fill=black] (0.838631756054356,-0.9006651355425708) circle (1.5pt); %punkt O
\draw[color=black] (0.9866666666666652,-0.72) node {$O$};
%\draw [fill=black] (4.084996261174071,2.4318024770562343) circle (1.5pt); %punkt P
%\draw[color=black] (4.226666666666665,2.72) node {$P$};
\draw [fill=black] (-0.7845770053117687,-0.9694730389841828) circle (1.5pt); %punkt Q
\draw[color=black] (-0.8133333333333348,-0.78) node {$Q$};
%\draw [fill=black] (6.633506230037735,-0.6550201065405512) circle (1.5pt); %punkt R
%\draw[color=black] (6.666666666666665,-0.48) node {$R$};
%\draw [fill=black] (2.7887211584967786,4.001885117652153) circle (1.5pt); %punkt S
%\draw[color=black] (2.8866666666666654,4.41) node {$S$};
%\draw [fill=uuuuuu] (2.783763484103525,3.4664562831808063) circle (1.5pt); %punkt T
%\draw[color=uuuuuu] (2.8866666666666654,3.14) node {$T$};
\end{scriptsize}
\end{tikzpicture}
   \caption{}
   \label{fig: conic2}
   \end{figure}
Analogously, using Theorem \ref{thm: Pascal} and taking curves $\alpha' =BD\cup CP\cup JL$ and $\beta' = BL\cup CD\cup JP$ one can prove that the points $B$,
$C$, $D$, $J$, $L$ and $P$ lie on a conic, and finally taking $\alpha'' = BD\cup CP\cup JL$ and $\beta'' = BC\cup DJ\cup LP$ one can prove that the points $M$,
$O$ and $R$ are collinear.

The collinearities of $M$, $O$, $Q$ and $M$, $O$, $R$ imply that the points $M$, $O$, $Q$ and $R$ lie on a line as asserted.
\end{proof}

The resulting configuration is indicated in Figure \ref{fig: whole
configuration}.

\begin{figure}[H]
\centering
\begin{tikzpicture}[line cap=round,line join=round,>=triangle 45,x=1.0cm,y=1.0cm]
%\clip(-6.233333333333335,-2.9) rectangle (14.726666666666665,9.08);
\clip(-3.233333333333335,-2.3) rectangle (9.726666666666665,9.08);
\draw [domain=-2.5:14.726666666666665] plot(\x,{(--9.541866666666666-1.54*\x)/3.26});%prosta DCKAN
\draw [domain=2.5:3.5] plot(\x,{(--5.9436-2.16*\x)/-0.02}); %prosta BEATSL
\draw [domain=-6.233333333333335:9.4] plot(\x,{(-3.4741333333333326--0.62*\x)/3.28}); %prosta FOBIC
\draw [domain=-6.233333333333335:7.5] plot(\x,{(--0.4227038325134407-3.1112790298592836*\x)/-5.052572932158789});
%prosta FGJAP
\draw [domain=-6.233333333333335:8.75] plot(\x,{(-3.5117772595537353--1.751710360158245*\x)/5.039984333365261});
%prosta FQHEK
\draw [domain=-6.233333333333335:14.726666666666665] plot(\x,{(-1.5778755345053685-0.5625375291001036*\x)/5.783293730432688});
%prosta DMBHG
\draw [domain=-6.233333333333335:14.726666666666665] plot(\x,{(--5.257943361515681-1.3629688593990654*\x)/5.775882329226216});
%prosta DRIEJ
\draw [domain=-6.233333333333335:14.726666666666665] plot(\x,{(--1.0317160740424687-1.0004833741555315*\x)/-0.21392961431237845});
%prosta LNJHO
\draw [domain=-6.233333333333335:14.726666666666665] plot(\x,{(--4.698918506243937-1.0224659349540723*\x)/0.21471685702482635});
%prosta LPKIM
\draw [domain=-6.233333333333335:5.2] plot(\x,{(-0.21610909610215334-2.4630053806498506*\x)/-1.7703517483201314});
%prosta SNGQ
\draw [domain=0:14.726666666666665] plot(\x,{(--14.32886322749831-2.3518024770562342*\x)/1.9416704054925944});
%prosta SPCR
\draw [domain=-6.233333333333335:14.726666666666665] plot(\x,{(-6.944919157563774--0.3144529324436316*\x)/7.418083235349504});
%prosta QOMR
%\draw [domain=-6.233333333333335:6] plot(\x,{(--2.8210648374795815--3.7072566054013603*\x)/3.790958064452835});
%prosta FNT
%\draw [domain=-1:14.726666666666665] plot(\x,{(-25.247040507035813--3.534340006156338*\x)/-4.444964135925282});
%prosta DPT
\begin{scriptsize}
\draw [fill=black] (2.7666666666666653,1.62) circle (1.5pt); %punkt A
\draw[color=black] (3.166666666666665,1.65) node {$A$};
\draw [fill=black] (2.7466666666666653,-0.54) circle (1.5pt); %punkt B
\draw[color=black] (2.8866666666666654,-0.3) node {$B$};
\draw [fill=black] (6.0266666666666655,0.08) circle (1.5pt); %punkt C
\draw[color=black] (6.166666666666665,0.36) node {$C$};
\draw [fill=black] (8.529960397099353,-1.1025375291001036) circle (1.5pt); %punkt D
\draw[color=black] (8.666666666666666,-0.92) node {$D$};
\draw [fill=black] (2.754078067873137,0.2604313302989618) circle (1.5pt); %punkt E
\draw[color=black] (2.8866666666666654,0) node {$E$};
\draw [fill=black] (-2.285906265492124,-1.4912790298592833) circle (1.5pt); %punkt F
\draw[color=black] (-2.373333333333335,-1.29) node {$F$};
\draw [fill=black] (-0.2652999493594206,-0.24702780510777403) circle (1.5pt); %punkt G
\draw[color=black] (-0.35333333333333486,0.06) node {$G$};
\draw [fill=black] (0.953056289709474,-0.3655365892428371) circle (1.5pt); %punkt H
\draw[color=black] (0.8066666666666651,-0.12) node {$H$};
\draw [fill=black] (4.634148432252167,-0.18321991016371605) circle (1.5pt); %punkt I
\draw[color=black] (4.766666666666665,0.1) node {$I$};
\draw [fill=black] (1.1669859040218524,0.6349467849126943) circle (1.5pt); %punkt J
\draw[color=black] (1.3666666666666651,0.97) node {$J$};
\draw [fill=black] (4.419431575227341,0.8392460247903564) circle (1.5pt); %punkt K
\draw[color=black] (4.306666666666665,0.55) node {$K$};
\draw [fill=black] (2.8295388752500106,8.410198527001244) circle (1.5pt); %punkt L
\draw[color=black] (3.046666666666665,8.42) node {$L$};
\draw [fill=black] (4.74999280834317,-0.7348623380116698) circle (1.5pt); %punkt M
\draw[color=black] (4.886666666666665,-0.56) node {$M$};
\draw [fill=black] (1.5050517989607108,2.2159775755420767) circle (1.5pt); %punkt N
\draw[color=black] (1.3266666666666653,2.5) node {$N$};
\draw [fill=black] (0.838631756054356,-0.9006651355425708) circle (1.5pt); %punkt O
\draw[color=black] (0.9866666666666652,-0.72) node {$O$};
\draw [fill=black] (4.084996261174071,2.4318024770562343) circle (1.5pt); %punkt P
\draw[color=black] (4.226666666666665,2.72) node {$P$};
\draw [fill=black] (-0.7845770053117687,-0.9694730389841828) circle (1.5pt); %punkt Q
\draw[color=black] (-0.8133333333333348,-0.78) node {$Q$};
\draw [fill=black] (6.633506230037735,-0.6550201065405512) circle (1.5pt); %punkt R
\draw[color=black] (6.666666666666665,-0.48) node {$R$};
\draw [fill=black] (2.7887211584967786,4.001885117652153) circle (1.5pt); %punkt S
\draw[color=black] (2.8866666666666654,4.41) node {$S$};
%\draw [fill=uuuuuu] (2.783763484103525,3.4664562831808063) circle (1.5pt); %punkt T
%\draw[color=uuuuuu] (2.8866666666666654,3.14) node {$T$};
\end{scriptsize}
\end{tikzpicture}
   \caption{}
   \label{fig: whole configuration}
   \end{figure}

\subsection{Algebraic proof}
   In this section we provide another
construction using algebraic methods quite in the spirit of Tao's
survey \cite{Tao2014}. This approach makes the study of the parameter
space, in particular degenerate cases, accessible. Since any two
lines on the projective plane are projectively equivalent we may
assume that points $A$, $B$ and $C$ are fundamental points, i.e.
$$A=(1:0:0), \textrm{ } B=(0:1:0), \textrm{ } C=(0:0:1).$$
Then we have the following equations of lines
$$AB: z=0, \textrm{ } AC: y=0, \textrm{ } BC: x=0.$$
On these lines, we choose points $E$, $D$, $F$ respectively, all distinct from the fundamental points. Thus we may assume that their
coordinates are:
$$ D=(a_2:0:a_1), \textrm{ } E=(b_1:b_2:0), \textrm{ } F=(0:c_1:c_2),$$
with $a_1,a_2,b_1,b_2,c_1,c_2 \neq 0$. Hence we obtain the following equations of lines
\begin{equation}\label{lines1}
   \begin{array}{ccc}
    AF & : & c_2y-c_1z=0,\\
    BD & : & a_1x-a_2z=0,\\
    EF & : & b_2c_2x-b_1c_2y+b_1c_1z=0,\\
    ED & : & a_1b_2x-a_1b_1y-a_2b_2z=0,\\
    \end{array}
   \end{equation}
which gives us the coordinates of points
\begin{equation}\label{p.coordinates1}
   \begin{array}{ccc}
    G = AF\cap BD & = & (a_2c_2:a_1c_1:a_1c_2),\\
    H = BD\cap EF & = & (a_2b_1c_2:a_2b_2c_2+a_1b_1c_1:a_1b_1c_2),\\
    I = BC\cap ED & = & (0:a_2b_2:-a_1b_1),\\
    J = AF\cap ED & = & (a_1b_1c_1+a_2b_2c_2:a_1b_2c_1:a_1b_2c_2),\\
    K = AC\cap EF & = & (b_1c_1:0:-b_2c_2).\\
   \end{array}
   \end{equation}
Then we obtain equations of the lines
\renewcommand\arraystretch{1.2}
 \begin{equation}\label{lines2}
   \begin{array}{ccc}
    HJ & : & a_1a_2 b_2^2 c_2^2 x + a_1^2b_1^2c_1 c_2 y - (a_1^2b_1^2c_1^2 + a_2^2b_2^2 c_2^2 +a_1a_2b_1b_2c_1c_2)z=0,\\
    IK & : & a_2b_2^2 c_2 x + a_1b_1^2c_1y + a_2b_1b_2c_1z=0,\\
    \end{array}
   \end{equation}
   and the point $L= HJ\cap IK =(a_1b_1^2c_1:-a_2b_2^2 c_2 : 0)$,
   hence it lies on the line through $A$ and $B$. This provides an alternative proof of Lemma \ref{lem: collinear1}.

    We find the coordinates of the remaining points
\renewcommand\arraystretch{1.2}
    \begin{equation}\label{p.coordinates2}
   \begin{array}{ccc}
    M = BD\cap IK & = & (-a_1a_2b_1^2c_1:a_2^2b_2^2 c_2 +a_1a_2b_1b_2c_1:-a_1^2b_1^2c_1),\\
    N = AC\cap HJ & = & (a_1^2b_1^2c_1^2 + a_2^2b_2^2 c_2^2 +a_1a_2b_1b_2c_1c_2:0:a_1a_2 b_2^2 c_2^2),\\
    O = HJ\cap BC & = & (0:a_1^2b_1^2c_1^2 + a_2^2b_2^2 c_2^2 +a_1a_2b_1b_2c_1c_2:a_1^2b_1^2c_1 c_2),\\
    P = AF\cap IK & = & (a_1b_1^2c_1^2+a_2b_1b_2c_1c_2:-a_2b_2^2c_1 c_2:- a_2b_2^2 c_2^2).
   \end{array}
   \end{equation}
   The remaining lines are
\renewcommand\arraystretch{1.2}
   \begin{equation}\label{lines3}
   \begin{array}{ccc}
    NG & : & a_1a_2 b_2^2 c_2^2 x + (a_1^2b_1^2c_1 c_2 +a_1a_2b_1 b_2 c_2^2) y - (a_1^2b_1^2c_1^2 + a_2^2b_2^2 c_2^2 +a_1 a_2b_1b_2c_1c_2)z=0,\\
    CP & : & a_2b_2^2 c_2 x + (a_1b_1^2c_1+a_2b_1b_2c_2)y=0,\\
    \end{array}
   \end{equation}
and finally we obtain the coordinates of points
\renewcommand\arraystretch{1.2}
    \begin{equation}\label{p.coordinates3}
   \begin{array}{ccc}
    Q = EF\cap NG & = & (a_2^2b_1b_2 c_2^2:a_1^2b_1^2c_1^2 +2a_1a_2b_1b_2c_1c_2 + a_2^2b_2^2 c_2^2:a_1^2b_1^2c_1 c_2 + 2a_1a_2b_1b_2 c_2^2),\\
    R = DE\cap CP & = & (a_1a_2b_1^2c_1 + a_2^2b_1b_2 c_2 :-a_2^2b_2^2 c_2: a_1^2b_1^2c_1 + 2a_1a_2b_1b_2c_2),\\
    S = CP\cap NG & = & (a_1b_1^2c_1+a_2b_1b_2c_2:-a_2b_2^2 c_2:0).
   \end{array}
   \end{equation}
Again, it is now obvious that $S$ is collinear with $A$ and $B$.
This gives an alternative proof of Lemma \ref{lem: collinear2}.

Checking the vanishing of determinants
\renewcommand\arraystretch{1.2}
  \[
  \begin{vmatrix}
  -a_2a_1b_1^2c_1 & a_2^2b_2^2 c_2 +a_1a_2b_1b_2c_1 & -a_1^2b_1^2c_1 \\
  0 & a_1^2b_1^2c_1^2 + a_2^2b_2^2 c_2^2 +a_1a_2b_1b_2c_1c_2 & a_1^2b_1^2c_1 c_2 \\
  a_2^2b_1b_2 c_2^2 & a_1^2b_1^2c_1^2 +2a_1a_2b_1b_2c_1c_2 + a_2^2b_2^2 c_2^2  & a_1^2b_1^2c_1 c_2 + 2a_2a_1b_1b_2 c_2^2
  \end{vmatrix} =0
  \]
  and
\renewcommand\arraystretch{1.2}
  \[
  \begin{vmatrix}
  -a_2a_1b_1^2c_1 & a_2^2b_2^2 c_2 +a_1a_2b_1b_2c_1 & -a_1^2b_1^2c_1 \\
  0 & a_1^2b_1^2c_1^2 + a_2^2b_2^2 c_2^2 +a_1a_2b_1b_2c_1c_2 & a_1^2b_1^2c_1 c_2 \\
  a_1a_2b_1^2c_1 + a_2^2b_1b_2 c_2 & -a_2^2b_2^2 c_2 & a_1^2b_1^2c_1 + 2a_1a_2b_1b_2c_2
  \end{vmatrix} = 0
  \]
  we obtain the collinearity of points $M, O, Q$ and $M, O, R$, what finally implies that all points $M, O, Q, R$ lie on a common line and
\renewcommand\arraystretch{1.2}
   \begin{equation}
   \begin{array}{ccc}
    MO & : & a_2(2a_1^2b_2^2c_2^2+a_2^2b_1^2c_1^2+2a_1a_2b_1b_2c_1c_2)x+a_1a_2^2b_1^2c_1c_2y-\\
    & & a_1(a_2^2b_1^2c_1^2+a_1^2b_2^2c_2^2+a_1a_2b_1b_2c_1c_2)z=0.\\
    \end{array}
   \end{equation}
This gives an alternative proof of Lemma \ref{lem: collinear3}.

\section{Parameter space for B\"or\"oczky configuration of $12$ lines and degenerations}
   The results presented in the previous section and this section put
   together give a proof of Theorem A.

   All configurations in which the points $D,E,F$ are mutually distinct from the points $A,B,C$
   are parametrized by $3$ parameters in $(\F^*)^3$, where $\F$ is the ground field (either $\R$ or $\Q$). For various
   reasons it is convenient to work with a compact parameter space. A natural compactification
   coming here into the picture is by $\calm=(\P^1(\F))^3$. This is a multi-homogeneous space
   with coordinates $(a_1:a_2)$, $(b_1:b_2)$ and $(c_1:c_2)$ introduced above.
   Our motivation to include degenerations stems also from the interest if they
   lead to containment counterexamples.

   We evaluate now the conditions that all points and lines appearing in the construction are distinct.
   Comparing the coordinates of all points and the equations of all lines we obtain the following
   degeneracy conditions.
\begin{itemize}
\item[i)] $a_1a_2b_1b_2c_1c_2= 0$
\item[ii)] $a_1b_1c_1+a_2b_2c_2= 0$,
\item[iii)]$a_1b_1c_1+2a_2b_2c_2= 0$
\end{itemize}
   This means that if the numbers $a_1,\dots,c_2$ do not satisfy any of above equations,
   then the points $A,\ldots,S$ defined in the previous section together with appropriate
   lines form a $B12$ configuration.
   The conditions i), ii) and iii) define divisors $D_1, D_2, D_3$ in $(\P^1)^3$ of three-degree $(2,2,2)$, $(1,1,1)$,
   and $(1,1,1)$ respectively. Their union contains all degenerate configurations.

   Below we present degenerations of our configuration corresponding to points in these divisors
   and their intersections.
\subsection{Degenerations coming from divisor $D_1$}
   We begin with the divisor $D_1$. This is the union of three pairs of disjoint ''planes'' ($\P^1\times\P^1$) in $\calm$.
   If the parameter $\frakm=((a_1:a_2),(b_1:b_2),(c_1:c_2))\in\calm$ is taken from $D_1$ but does not belong
   to its singular locus, then just one coordinate is zero. Without loss of generality we may assume
   that it is $a_1=0$, i.e. $D=A$. Then the configuration degenerates to the configuration
   of $7$ lines with $6$ triple points. These incidences are indicated in Figure \ref{fig: degeneracja_1_1}.

  \begin{figure}[H]
  \centering
\begin{tikzpicture}[line cap=round,line join=round,>=triangle 45,x=1.0cm,y=1.0cm,scale=0.5]
\clip(-4.3,-5.68) rectangle (16.66,6.3);
\draw [domain=-4.3:16.66] plot(\x,{(--12.5556-4.74*\x)/-2.14});
\draw [domain=-4.3:16.66] plot(\x,{(-12.0996--2.34*\x)/5.74});
\draw [domain=-4.3:16.66] plot(\x,{(--22.656-2.4*\x)/3.6});
\draw [domain=-4.3:16.66] plot(\x,{(--15.268267104102089-4.1115455800167435*\x)/-0.59840667918637});
\draw [domain=-4.3:16.66] plot(\x,{(--11.391297635285644-2.6618182557289725*\x)/-2.7729976656180257});
\draw [domain=-4.3:16.66] plot(\x,{(--23.11624806668627-3.8424603757694564*\x)/-6.418325992435999});
\draw [domain=-4.3:16.66] plot(\x,{(-1.2613065060610236-0.25000284285084473*\x)/1.4137400152307782});
\begin{scriptsize}
\draw [fill=black] (4.22,3.48) circle (1.5pt);
\draw[color=black] (3.86,3.38) node {$A$};
\draw [fill=black] (7.82,1.08) circle (1.5pt);
\draw[color=black] (7.88,0.88) node {$B$};
\draw [fill=black] (2.08,-1.26) circle (1.5pt);
\draw[color=black] (1.94,-1.) node {$C$};
\draw [fill=black] (6.3945909864316555,2.0302726757122294) circle (1.5pt);
\draw[color=black] (6.44,1.78) node {$E$};
\draw [fill=black] (3.6215933208136297,-0.6315455800167433) circle (1.5pt);
\draw[color=black] (3.8,-0.74) node {$F$};
\draw [fill=black] (1.4016740075640006,-2.7624603757694564) circle (1.5pt);
\draw[color=black] (1.54,-2.9) node {$K$};
\draw [fill=black] (3.4937400152307783,-1.5100028428508447) circle (1.5pt);
\draw[color=black] (3.72,-1.7) node {$P$};
\draw [fill=black] (14.66943396226417,-3.486289308176113) circle (1.5pt);
\draw[color=black] (14.8,-3.2) node {$R$};
\end{scriptsize}
\end{tikzpicture}
   \caption{}
   \label{fig: degeneracja_1_1}
   \end{figure}

   Let now $\frakm$ be a double point on $D_1$. Without loss of generality, we may assume
   that it is defined by $a_1=b_1=0$. Then the whole configuration degenerates to a quasi-pencil
   on four lines, i.e. there are $4$ lines and one triple point. The incidences are
   indicated in  Figure \ref{fig: degeneracja_1_2}.
  \begin{figure}[H]
  \centering
  \begin{tikzpicture}[line cap=round,line join=round,>=triangle 45,x=1.0cm,y=1.0cm,scale=0.5]
\clip(-5.32,-5.84) rectangle (15.64,6.14);
\draw [domain=-5.32:15.64] plot(\x,{(--16.6992-5.88*\x)/1.86});
\draw [domain=-5.32:15.64] plot(\x,{(-27.6784--3.94*\x)/2.96});
\draw [domain=-5.32:15.64] plot(\x,{(--13.754-1.94*\x)/4.82});
\draw [domain=-5.32:15.64] plot(\x,{(--11.573634226447709--0.9768277698614787*\x)/7.011322385479689});
\begin{scriptsize}
\draw [fill=black] (2.22,1.96) circle (1.5pt);
\draw[color=black] (2.36,2.24) node {$A$};
\draw [fill=black] (7.04,0.02) circle (1.5pt);
\draw[color=black] (7.12,-0.26) node {$B$};
\draw [fill=black] (4.08,-3.92) circle (1.5pt);
\draw[color=black] (4.32,-3.84) node {$C$};
\draw [fill=black] (9.23132238547969,2.9368277698614786) circle (1.5pt);
\draw[color=black] (9.38,2.76) node {$F$};
\end{scriptsize}
\end{tikzpicture}
   \caption{}
   \label{fig: degeneracja_1_2}
   \end{figure}

   Finally, if $\frakm$ is a triple point on $D_1$, then everything reduces to the triangle
   with vertices $A,B,C$.

\subsection{Degenerations associated to $D_2$}

   This case has an easy geometric interpretation, namely it happens when the points $D,E,F$ are collinear,
   i.e. to begin with we get a Menelaus configuration. It implies that some points and lines in the generic
   configuration coincide. In fact there are only six lines left. They intersect in $4$ triple and $3$ double
   points. This is indicated in Figure \ref{fig: degeneracja_2}.

  \begin{figure}[H]
  \centering
  \begin{tikzpicture}[line cap=round,line join=round,>=triangle 45,x=1.0cm,y=1.0cm,scale=0.2]
\clip(-14.01,-11.35) rectangle (33.15,14.48);
\draw [domain=3.5:5] plot(\x,{(--26.450325-6.345*\x)/-0.18});
\draw [domain=-14.01:33.15] plot(\x,{(--53.294625-8.01*\x)/3.915});
\draw [domain=-14.01:33.15] plot(\x,{(-0.5618249999999989--1.665*\x)/-4.095});
\draw [domain=-14.01:33.15] plot(\x,{(--3.647729597766448-0.558347545084886*\x)/2.4847052311637965});
\draw [domain=-14.01:33.15] plot(\x,{(-48.55863717015044-1.6926186001013068*\x)/-11.622343442994083});
\draw [domain=-14.01:33.15] plot(\x,{(-10.1438844107281--1.5095985558325449*\x)/2.5433703687784757});
\begin{scriptsize}
\draw [fill=black] (4.305,4.805) circle (1.5pt);
\draw[color=black] (4.62,5.435) node {$A$};
\draw [fill=black] (4.125,-1.54) circle (1.5pt);
\draw[color=black] (4.44,-0.91) node {$B$};
\draw [fill=black] (8.22,-3.205) circle (1.5pt);
\draw[color=black] (8.535,-2.575) node {$C$};
\draw [fill=black] (6.668370368778476,-0.03040144416745516) circle (1.5pt);
\draw[color=black] (7.005,0.62) node {$D$};
\draw [fill=black] (4.183665137614679,0.5279461009174309) circle (1.5pt);
\draw[color=black] (4.485,1.16) node {$E$};
\draw [fill=black] (-7.317343442994084,3.112381399898693) circle (1.5pt);
\draw[color=black] (-6.99,3.725) node {$F$};
\draw [fill=black] (18.23233960191375,6.833306440585946) circle (1.5pt);
\draw[color=black] (18.525,7.46) node {$G$};
\end{scriptsize}
\end{tikzpicture}
   \caption{}
   \label{fig: degeneracja_2}
   \end{figure}

\subsection{Degenerations associated to divisor $D_3$}
   The last degenerating condition is $iii)$. Interesting phenomena in this case is that
   points $E,Q,R,S$ coincide so that the point $E$ turns to a sixtuple point, there are $15$ triple
   points in this configuration and $6$ double points, which are not named in Figure \ref{fig: degeneracja_4}
   indicating the incidences in this case.
  \begin{figure}[H]
  \centering
\begin{tikzpicture}[line cap=round,line join=round,>=triangle 45,x=1.0cm,y=1.0cm,scale=0.5]
\clip(-4.4661028893587,-5.795193798449621) rectangle (16.981338971106428,6.463410852713181);
\draw [domain=-4.4661028893587:16.981338971106428] plot(\x,{(--7.2116--3.4866666666666513*\x)/10.566666666666675});
\draw [domain=-4.4661028893587:16.981338971106428] plot(\x,{(--22.9396-6.54*\x)/3.26});
\draw [domain=-4.4661028893587:16.981338971106428] plot(\x,{(--96.20053333333328-10.026666666666651*\x)/-7.3066666666666755});
\draw [domain=-4.4661028893587:16.981338971106428] plot(\x,{(--35.93326941080506-7.15840044824718*\x)/1.3858800947002186});
\draw [dash pattern=on 3pt off 3pt,domain=-4.4661028893587:16.981338971106428] plot(\x,{(--6.554660188041995-1.1250859791106933*\x)/2.211662230924625});
\draw [domain=-4.4661028893587:16.981338971106428] plot(\x,{(--18.94279767913543-4.611752645777345*\x)/-8.35500443574205});
\draw [domain=-4.4661028893587:16.981338971106428] plot(\x,{(--25.277256549465793-3.5970107676784018*\x)/3.9811481066040697});
\draw [domain=-4.4661028893587:16.981338971106428] plot(\x,{(--6.632371598286905-0.5795477193181251*\x)/3.200001140342788});
\draw [domain=-4.4661028893587:16.981338971106428] plot(\x,{(--17.034015923467628-2.6043867421084266*\x)/3.78897418611664});
\draw [domain=-4.4661028893587:16.981338971106428] plot(\x,{(-28.25062441853348--5.331015093309159*\x)/-4.308195703751384});
\draw [domain=-4.4661028893587:16.981338971106428] plot(\x,{(--30.54044832941373-2.8753615381914415*\x)/7.883319285934947});
\draw [domain=-4.4661028893587:16.981338971106428] plot(\x,{(-5.844079982481696--1.5633635913192028*\x)/0.608715219597856});
\draw [domain=-4.4661028893587:16.981338971106428] plot(\x,{(-20.880090352473573--3.5395997422338588*\x)/-2.1009569587586405});
\begin{scriptsize}
\draw [fill=black] (2.72,1.58) circle (1.5pt);
\draw[color=black] (2.8604087385482813,1.8587596899224788) node {$A$};
\draw [fill=black] (13.286666666666676,5.066666666666651) circle (1.5pt);
\draw[color=black] (13.420408738548286,4.928527131782946) node {$B$};
\draw [fill=black] (5.98,-4.96) circle (1.5pt);
\draw[color=black] (6.196222692036654,-4.874263565891481) node {$D$};
\draw [fill=black] (4.594119905299782,2.19840044824718) circle (1.5pt);
\draw[color=black] (4.7431994362227,2.493178294573642) node {$E$};
\draw [fill=black] (4.931662230924625,0.45491402088930677) circle (1.5pt);
\draw[color=black] (5.070641296687817,0.7331782945736404) node {$I$};
\draw [fill=black] (8.575268011903852,-1.3986103194312218) circle (1.5pt);
\draw[color=black] (8.836222692036657,-1.1905426356589195) node {$H$};
\draw [fill=black] (3.6370317630651496,-0.2596894878668947) circle (1.5pt);
\draw[color=black] (3.8427343199436304,-0.26961240310077916) node {$C$};
\draw [fill=black] (5.920001140342788,1.000452280681875) circle (1.5pt);
\draw[color=black] (6.073431994362236,1.2857364341085247) node {$F$};
\draw [fill=black] (9.810111768901647,0.29591979966063997) circle (1.5pt);
\draw[color=black] (10.084594785059913,0.5285271317829425) node {$G$};
\draw [fill=black] (4.786293825787212,1.2057764226772048) circle (1.5pt);
\draw[color=black] (4.927385482734328,1.4903875968992226) node {$J$};
\draw [fill=black] (0.6234665271732415,5.785929114198466) circle (1.5pt);
\draw[color=black] (0.7729668780831634,6.074573643410854) node {$K$};
\draw [fill=black] (3.7482455411103297,1.9192885917985474) circle (1.5pt);
\draw[color=black] (3.88366455250177,2.2066666666666648) node {$L$};
\draw [fill=black] (7.557854485432987,-2.794769027289046) circle (1.5pt);
\draw[color=black] (7.792501761804098,-2.725426356589154) node {$M$};
\draw [fill=black] (1.9267924829667,3.1712813378520814) circle (1.5pt);
\draw[color=black] (2.0622692036645596,3.4550387596899217) node {$N$};
\draw [fill=black] (5.4568975266743465,0.7448307149448125) circle (1.5pt);
\draw[color=black] (5.725525017618049,0.7945736434108497) node {$O$};
\draw [fill=black] (4.245746982663006,1.303674103452308) circle (1.5pt);
\draw[color=black] (4.006455250176189,1.244806201550385) node {$P$};
\end{scriptsize}
\end{tikzpicture}
   \caption{}
   \label{fig: degeneracja_4}
   \end{figure}
   Note that, compared to the general case, there are additional
   collinearities of points: $NEG$, $CEP$, $EMO$ and $AIH$.
   The last mentioned line is not a line of the original configuration.
   We present incidences of this degeneration in Table \ref{tab:config iv}.
\begin{table}
\begin{tabular}{|c|c|c|c|c|c|c|c|c|c|c|c|c|}
  \hline
    & AC & AB & BC & AF & BD & EF & DE & HJ & IK & EG & CE & EM \\
  \hline
  A & + & + &  & + &  &  &  &  &  &  &   & \\
  \hline
  B &  & + & + &  & + &  &  &  &  &  &  &  \\
  \hline
  C & + &  & + &  &  &  &  &  &  &  & + & \\
  \hline
  D & + &  &  &  & + &  & + &  &  &  &  & \\
  \hline
  E &  & + &  &  &  & + & + &  &  & + & + & +\\
  \hline
  F &  &  & + & + &  & + &  &  &  &  &  & \\
  \hline
  G &  &  &  & + & + &  &  &  &  & + &  & \\
  \hline
  H &  &  &  &  & + & + &  & + &  &  &  & \\
  \hline
  I &  &  & + &  &  &  & + &  & + &  &  & \\
  \hline
  J &  &  &  & + &  &  & + & + &  &  &  & \\
  \hline
  K & + &  &  &  &  & + &  &  & + &  &  & \\
  \hline
  L &  & + &  &  &  &  &  & + & + &  &  & \\
  \hline
  M &  &  &  &  & + &  &  &  & + &  &  & +\\
  \hline
  N & + &  &  &  &  &  &  & + &  & + &  & \\
  \hline
  O &  &  & + &  &  &  &  & + &  &  &  & +\\
  \hline
  P &  &  &  & + &  &  &  &  & + &  &  + & \\
  \hline
\end{tabular}
\caption{: Incidence table for degeneration iii)}
\label{tab:config iv}
\end{table}

   The incidences between the divisors $D_1,\ldots,D_3$ are explained in the next remark.
 \begin{remark}
 The intersection locus $Z$ of any two divisors $D_i$, $D_j$ is contained also in
 the third divisor $D_k$, for $\{i,j,k\}=\{1,2,3\}$ and this set $Z$ is contained
 in the singular locus of divisor $D_1$. The degenerations corresponding to points in $Z$
 (away of triple points of $D_1$) are indicated in Figure \ref{fig: degeneracja_1_2}.
 \end{remark}
In this case both conditions are equivalent to
$a_1b_1c_1=a_2b_2c_2$.

We conclude this section by the following observation.
\begin{corollary}
   None of degenerate configurations provides a counterexample to the
   containment problem. This shows in particular that being a counterexample
   configuration is not a closed condition.
\end{corollary}
\proof
   Since we have concrete coordinates to work with, the claim can be easily checked
   with a symbolic algebra program, we used Singular \cite{DGPS}.
\endproof
\section{Constructions of $15$ lines with $31$ triple points}
   The configuration we are now interested in is visualized in Figure \ref{fig: B15}.
   For clarity the points in the figure are labeled by numbers only. The number $i$ in the picture corresponds
   to the point $P_i$ in the text below.
  \begin{figure}[H]
  \centering
\begin{tikzpicture}[line cap=round,line join=round,>=triangle 45,x=1.0cm,y=1.0cm,scale=0.5]
\clip(-8.9,-13.67) rectangle (18.83,10.23);
\draw [domain=-6.9:18.83] plot(\x,{(-41.63--9.51*\x)/1});    %15  5
\draw [domain=-8.9:18.83] plot(\x,{(-52.35--7.36*\x)/2.39}); %13  7
\draw [domain=1.9:18.83] plot(\x,{(-23.44--2.56*\x)/1.48});  %25 23
\draw [domain=.9:18.83] plot(\x,{(--21.32-2.2*\x)/-1.98});   %24 22
\draw [domain=-.9:18.83] plot(\x,{(--35.9-4.55*\x)/-6.26});  %12 10
\draw [domain=-3:18.83] plot(\x,{(--8.72-3.89*\x)/-8.74});   %19  9
\draw [domain=-5:18.83] plot(\x,{(-22.1-1.61*\x)/-7.57});    %21  6
\draw [domain=-6.9:18] plot(\x,{(-16.4-0*\x)/-2.96});
\draw [domain=-8:16] plot(\x,{(--19.09-0.61*\x)/2.89});      %16 30
\draw [domain=-8.9:14] plot(\x,{(--42.91-3.15*\x)/7.07});    %20  7
\draw [domain=-8.9:11] plot(\x,{(--32.91-5.62*\x)/7.74});    % 8 18
\draw [domain=-8.9:8] plot(\x,{(--9.17-5.75*\x)/5.18});      %11 10
\draw [domain=-8.9:5.5] plot(\x,{(-0.51-2.56*\x)/1.48});     %23 22
\draw [domain=-8.9:4] plot(\x,{(-0.06--2.81*\x)/-0.91});     %25 24
\draw [domain=-8.9:18.83] plot(\x,{(-13.53--7.7*\x)/-0.81}); % 6  9
\begin{scriptsize}
\fill [color=uuuuuu] (-6.46,8.95) circle (3pt);
\draw[color=uuuuuu] (-6.2,9.3) node {$_1$};
\fill [color=uuuuuu] (-4.57,7.57) circle (3pt);
\draw[color=uuuuuu] (-4.2,8.25) node {$_{27}$};
\fill [color=uuuuuu] (-1.78,5.55) circle (3pt);
\draw[color=uuuuuu] (-1.3,6) node {$_{31}$};
\fill [color=uuuuuu] (1.42,3.22) circle (3pt);
\draw[color=uuuuuu] (1.6,3.7) node {$_6$};
\fill [color=uuuuuu] (4.48,1) circle (3pt);
\draw[color=uuuuuu] (4.29,1.6) node {$_4$};
\fill [color=uuuuuu] (6.87,-0.74) circle (3pt);
\draw[color=uuuuuu] (6.86,-0.1) node {$_8$};
\fill [color=uuuuuu] (-3.4,5.55) circle (3pt);
\draw[color=uuuuuu] (-3.7,5.2) node {$_{17}$};
\fill [color=uuuuuu] (-0.87,2.73) circle (3pt);
\draw[color=uuuuuu] (-0.47,3.17) node {$_{21}$};
\fill [color=uuuuuu] (1.78,-0.21) circle (3pt);
\draw[color=uuuuuu] (2,0.4) node {$_9$};
\fill [color=uuuuuu] (4.08,-2.77) circle (3pt);
\draw[color=uuuuuu] (4.6,-2.75) node {$_{10}$};
\fill [color=uuuuuu] (5.65,-4.5) circle (3pt);
\draw[color=uuuuuu] (5.2,-4.5) node {$_{11}$};
\fill [color=uuuuuu] (0.3,-0.86) circle (3pt);
\draw[color=uuuuuu] (0.56,-0.3) node {$_{19}$};
\fill [color=uuuuuu] (2.19,-4.14) circle (3pt);
\draw[color=uuuuuu] (2.7,-4.2) node {$_{12}$};
\fill [color=uuuuuu] (3.67,-6.7) circle (3pt);
\draw[color=uuuuuu] (4.25,-6.62) node {$_{22}$};
\fill [color=uuuuuu] (4.48,-8.1) circle (3pt);
\draw[color=uuuuuu] (5.1,-8) node {$_{23}$};
\fill [color=uuuuuu] (2.59,-7.9) circle (3pt);
\draw[color=uuuuuu] (2.1,-7.68) node {$_{24}$};
\fill [color=uuuuuu] (3.31,-10.13) circle (3pt);
\draw[color=uuuuuu] (2.5,-10.1) node {$_{25}$};
\fill [color=uuuuuu] (3.07,-12.45) circle (3pt);
\draw[color=uuuuuu] (3.55,-12.4) node {$_3$};
\fill [color=uuuuuu] (-2.28,7.09) circle (3pt);
\draw[color=uuuuuu] (-2,7.6) node {$_{29}$};
\fill [color=uuuuuu] (1.17,5.55) circle (3pt);
\draw[color=uuuuuu] (0.8,5.2) node {$_{14}$};
\fill [color=uuuuuu] (4.79,3.94) circle (3pt);
\draw[color=uuuuuu] (4.61,4.4) node {$_5$};
\fill [color=uuuuuu] (7.94,2.54) circle (3pt);
\draw[color=uuuuuu] (7.89,3) node {$_7$};
\fill [color=uuuuuu] (10.07,1.58) circle (3pt);
\draw[color=uuuuuu] (10,2.17) node {$_{20}$};
\fill [color=uuuuuu] (4.96,5.55) circle (3pt);
\draw[color=uuuuuu] (5.4,6) node {$_{15}$};
\fill [color=uuuuuu] (8.66,4.76) circle (3pt);
\draw[color=uuuuuu] (8.4,5.3) node {$_{13}$};
\fill [color=uuuuuu] (11.55,4.15) circle (3pt);
\draw[color=uuuuuu] (11.4,4.6) node {$_{30}$};
\fill [color=uuuuuu] (13.13,3.81) circle (3pt);
\draw[color=uuuuuu] (13.3,3.45) node {$_{16}$};
\fill [color=uuuuuu] (12.36,5.55) circle (3pt);
\draw[color=uuuuuu] (12.2,6) node {$_{28}$};
\fill [color=uuuuuu] (14.7,5.55) circle (3pt);
\draw[color=uuuuuu] (14.6,6.4) node {$_{26}$};
\fill [color=uuuuuu] (16.83,6.5) circle (3pt);
\draw[color=uuuuuu] (16.9,6.15) node {$_2$};
\fill [color=uuuuuu] (8.18,-1.69) circle (3pt);
\draw[color=uuuuuu] (8.75,-1.65) node {$_{18}$};
\end{scriptsize}
\end{tikzpicture}
   \caption{}
   \label{fig: B15}
   \end{figure}

\subsection{Construction}
   In this section we construct a configuration of   $15$ lines with $31$ triple points using the algebraic method.
   To begin with let $P_1$, $P_2$, $P_3$ and $P_4$ be the standard points, i.e.
   $$ P_1=(1:0:0),\;\; P_2=(0:1:0), \;\; P_3=(0:0:1), \;\; P_4=(1:1:1).$$
   Then we have the following equations of lines
   $$P_1P_4: -y+z=0, \;\; P_2P_4: x-z=0, \;\; P_3P_4: -x+y=0.$$
   On the line $P_3P_4$  we choose a point $P_5$ distinct from the fundamental points. Thus we may assume that its
   coordinates are:
   $$ P_5=(a:a:1)$$
   with $a\notin\{0,1\}$. Hence we obtain the following equations of lines
 \begin{equation}
   \begin{array}{ccc}
    P_1P_5 & : & -y+az=0,\\
    P_2P_5 & : & x-az=0,\\
    \end{array}
 \end{equation}
which gives us coordinates of the points
 \begin{equation*}
   \begin{array}{ccc}
    P_6 = P_1P_4\cap P_2P_5 & = & (a:1:1),\\
    P_7 = P_2P_4\cap P_1P_5 & = & (1:a:1).\\
   \end{array}
   \end{equation*}
Then we obtain the following equations of the lines and coordinates of points
 \begin{equation*}
   \begin{array}{ccc}
    P_3P_6 & : & -x+ay=0,\\
    P_3P_7 & : & ax-y=0,\\
   \end{array}
   \end{equation*}
 \begin{equation*}
   \begin{array}{ccc}
    P_8 = P_1P_4\cap P_3P_7 & = & (1:a:a),\\
   \end{array}
   \end{equation*}
 \begin{equation*}
   \begin{array}{ccc}
    P_2P_8 & : & ax-z=0,\\
   \end{array}
   \end{equation*}
 \begin{equation*}
   \begin{array}{ccc}
    P_9 = P_2P_4\cap P_3P_6 & = & (a:1:a),\\
   \end{array}
   \end{equation*}
 \begin{equation*}
   \begin{array}{ccc}
    P_1P_9 & : & -ay+z=0,\\
   \end{array}
   \end{equation*}
 \begin{equation*}
   \begin{array}{ccc}
    P_{10} = P_2P_8\cap P_3P_4 & = & (1:1:a).\\
   \end{array}
   \end{equation*}
\begin{fact}
   The incidence $P_{10}\in P_1P_9$ follows from above choices.
\end{fact}
   Further we obtain coordinates of the points
 \begin{equation*}
   \begin{array}{ccc}
    P_{11} = P_3P_7\cap P_1P_9 & = & (1:a:a^2),\\
    P_{12} = P_3P_6\cap P_2P_8 & = & (a:1:a^2),\\
    P_{13} = P_2P_5\cap P_3P_7 & = & (a:a^2:1),\\
    P_{14} = P_1P_5\cap P_3P_6 & = & (a^2:a:1).\\
   \end{array}
   \end{equation*}
Now we encounter the second choice in this construction. On the line
$P_3P_4$ we choose a point $P_{15}$ distinct from the fundamental
points and the point $P_5$. Thus we may assume that its coordinates
are:
$$ P_{15}=(b:b:1)$$
with $b\notin\{0,1,a\}$.

Then we obtain the following equations of lines and coordinates of points
\begin{equation*}
   \begin{array}{ccc}
    P_{14}P_{15} & : & (-a+b)x+(a^2-b)y+(-a^2b+ab)z=0,\\
    P_{13}P_{15} & : & (-a^2+b)x+(a-b)y+(a^2b-ab)z=0,\\
   \end{array}
   \end{equation*}
 \begin{equation*}
   \begin{array}{ccc}
    P_{16} = P_2P_8\cap P_{13}P_{15} & = & (a-b:-a^3b+a^2b+a^2-b:a^2-ab),\\
    P_{17} = P_1P_9\cap P_{14}P_{15} & = & (a^3b-a^2b-a^2+b:-a+b:-a^2+ab),\\
   \end{array}
   \end{equation*}
 \begin{equation*}
   \begin{array}{ccc}
    P_{11}P_{16} & : & (-a^5b+a^4b+a^4-a^3)x+(-a^3+a^2b+a^2-ab)y+\\
    & & (a^3b-a^2b-ab+b)z=0,\\
    P_{12}P_{17} & : & (-a^3+a^2b+a^2-ab)x+(-a^5b+a^4b+a^4-a^3)y+\\
    & & (a^3b-a^2b-ab+b)z=0,\\
   \end{array}
   \end{equation*}
 \begin{equation*}
   \begin{array}{ccl}
    P_{18} & = & P_1P_4\cap P_{11}P_{16} = \\
    & = & (-a^3b+a^3-a^2+2ab-b:-a^5b+a^4b+a^4-a^3:-a^5b+a^4b+a^4-a^3),\\
    P_{19} & = & P_2P_4\cap P_{12}P_{17} = \\
    & = & (-a^5b+a^4b+a^4-a^3:-a^3b+a^3-a^2+2ab-b :-a^5b+a^4b+a^4-a^3),\\
   \end{array}
   \end{equation*}
 \begin{equation*}
   \begin{array}{ccc}
    P_{20} = P_1P_5\cap P_2P_8 & = & (1:a^2:a),\\
    P_{21} = P_2P_5\cap P_1P_9 & = & (a^2:1:a),\\
   \end{array}
   \end{equation*}
 \begin{equation*}
   \begin{array}{ccc}
    P_{18}P_{20} & : & (a^7b-2a^6b-a^6+a^5b-a^4)x+(-a^5b+2a^4b-2a^2b+ab)y+\\
    & & (a^5-a^4b-2a^4+2a^3b+a^3-a^2b)z=0,\\
    P_{19}P_{21} & : & (-a^5b+2a^4b-2a^2b+ab)x+(a^7b-2a^6b-a^6+a^5b-a^4)y+\\
    & & (a^5-a^4b-2a^4+2a^3b+a^3-a^2b)z=0,\\
   \end{array}
   \end{equation*}
 \begin{equation*}
   \begin{array}{ccl}
    P_{22} = P_3P_4\cap P_{11}P_{16} & = & (a^3b-a^2b-ab+b:a^3b-a^2b-ab+b:\\
    &  & a^5b-a^4b-a^4+2a^3-a^2b-a^2+ab).\\
   \end{array}
   \end{equation*}
\begin{fact}
Note that the incidence $P_{22}\in P_{12}P_{17}$ does not impose any
additional conditions on $a$ and $b$.
\end{fact}
The coordinates of the remaining points are now easy to find:
\begin{equation*}
   \begin{array}{ccl}
    P_{23} = P_3P_7\cap P_{12}P_{17} & = & (-a^3b+a^2b+ab-b:-a^4b+a^3b+a^2b-ab:\\
    & & -a^6b+a^5b+a^5-a^4-a^3+a^2b+a^2-ab),\\
    P_{24} = P_3P_6\cap P_{11}P_{16} & = & (-a^4b+a^3b+a^2b-ab:-a^3b+a^2b+ab-b:\\
    & & -a^6b+a^5b+a^5-a^4-a^3+a^2b+a^2-ab),\\
    P_{25} = P_3P_4\cap P_{18}P_{20} & = & (a^5-a^4b-2a^4+2a^3b+a^3-a^2b:\\
    & & a^5-a^4b-2a^4+2a^3b+a^3-a^2b:\\
    & & -a^7b+2a^6b+a^6-2a^5-2a^4b+a^4+2a^2b-ab),\\
    P_{26} = P_2P_4\cap P_{11}P_{16} & = & (-a^3+a^2b+a^2-ab:\\
    & & a^5b-a^4b-a^4-a^3b+a^3+a^2b+ab-b:\\
    & & -a^3+a^2b+a^2-ab),\\
    P_{27} = P_1P_4\cap P_{12}P_{17} & = & (a^5b-a^4b-a^4-a^3b+a^3+a^2b+ab-b:\\
    & & -a^3+a^2b+a^2-ab:-a^3+a^2b+a^2-ab),\\
    P_{28} = P_2P_5\cap P_{14}P_{15} & = & (a^3-ab:a^2b+a^2-2ab:a^2-b),\\
    P_{29} = P_1P_5\cap P_{13}P_{15} & = & (a^2b+a^2-2ab:a^3-ab:a^2-b),\\
    P_{30} = P_2P_4\cap P_{13}P_{15} & = & (a-b:-a^2b+a^2+ab-b:a-b),\\
    P_{31} = P_1P_4\cap P_{14}P_{15} & = & (-a^2b+a^2+ab-b:a-b:a-b).\\
   \end{array}
   \end{equation*}
\begin{fact}
It is easy to see that $P_{25}\in P_{19}P_{21}$.
\end{fact}
   Finally we have to check under which conditions the following incidences are satisfied
 \begin{equation}\label{conditions}
   \begin{array}{cccc}
   P_{23}\in P_{18}P_{20}, & P_{24}\in P_{19}P_{21}, & P_{26}\in P_{14}P_{15}, & P_{27}\in P_{13}P_{15},\\
   P_{28}\in P_{18}P_{20}, & P_{29}\in P_{19}P_{21}, & P_{30}\in P_{18}P_{20}, & P_{31}\in P_{19}P_{21}.\\
   \end{array}
   \end{equation}
   Evaluating algebraic conditions we obtain polynomial equations involving $a$ and $b$ and
   the smallest ideal that contains all these polynomials is the ideal
   generated by the polynomial
   $$(a-1)^2(a^4b-a^2b^2-a^3+a^2b-ab^2+b^2).$$
   Since by assumption $a\neq 1$, it must be
\begin{equation}\label{podstawienie b}
   f:=a^4b-a^2b^2-a^3+a^2b-ab^2+b^2=0.
\end{equation}
\subsection{The parameter curve}
   The polynomial $f$ in \eqnref{podstawienie b} defines a singular curve in the plane $\R^2$ with coordinates $(a,b)$.
   We want to pass to its smooth model. To this end we substitute
   $$ b=\frac{(a-1)aT+a^2+a^4}{2(a^2+a-1)}$$
   into the equation (\ref{podstawienie b}) and we get
   $$ (a-1)^2a^2(-a^4-2a^3-5a^2+T^2-4a)=0.$$
   We can again localize at $a=1$, so that it is enough to study the curve
   $$ C:\;\; T^2=a (1 + a) (4 + a + a^2). $$
   Performing additional substitutions
   $$a=\frac1X, \; \; T=\frac{2Y+X+1}{X^2}$$
   we obtain a smooth elliptic curve $E$ in the canonical form
   $$ E: Y^2+XY+Y=X^3+X^2.$$
   This is the parameter space for $B15$ configurations and thus Theorem B is proved.

   It is known (see Cremona basis \cite{Cre}) that $E$ contains only
   $4$ rational points. Each of them corresponds to forbidden values of
   $a$ and $b$. Thus Corollary \ref{cor:corollary B} follows.
\paragraph*{\emph{Acknowledgement.}}
   We would like to thank Maciej Ulas for helping us with arithmetic questions.
   We thank also Tomasz Szemberg for valuable conversations.
%*****************************************************************************

%***************************************************************************** % Addresses

\bigskip \small

\bigskip
   Magdalena~Lampa-Baczy\'nska, Justyna Szpond,
   Institute of Mathematics
   Pedagogical University of Cracow,
   Podchor\c a\.zych 2,
   PL-30-084 Krak\'ow, Poland

\nopagebreak

  \textit{E-mail address:} \texttt{lampa.baczynska@wp.pl}

%   \textit{E-mail address:} \texttt{szemberg@up.krakow.pl}

   \textit{E-mail address:} \texttt{szpond@up.krakow.pl}

\bigskip

%   Justyna Szpond current address:
%   Albert-Ludwigs-Universit\"at Freiburg,
%   Mathematisches Institut, D-79104 Freiburg, Germany
%*****************************************************************************

\end{document}